\documentclass[a4paper]{amsart}
\usepackage{a4,pstricks,bbm}
\usepackage{amsmath,amssymb,latexsym}
\usepackage{epsf,psfrag,epsfig}
\usepackage[english]{babel}
\usepackage[latin1]{inputenc}

\newcommand{\Ref}[1]{(\ref{#1})}

\newcommand{\findem}{\hfill $\square$\\ \medskip}

\newtheorem{thm}{Theorem}
\newtheorem{prop}[thm]{Proposition}
\newtheorem{lemma}[thm]{Lemma}

\newtheorem{Def}[thm]{Definition}

\newcommand{\pare}[1]{\left( #1 \right)}
\newcommand{\parfrac}[2]{\left( \frac{#1}{#2} \right)}
\newcommand{\ite}{~\smallskip \\ \noindent $\bullet$ }

\newcommand{\al}{\alpha}
\newcommand{\be}{\beta}
\newcommand{\ga}{\gamma}
\newcommand{\De}{\Delta}
\newcommand{\Dee}{(\Delta\cup\{1\})}
\newcommand{\de}{\delta}
\newcommand{\ka}{\kappa}

\newcommand{\bb}[1]{\overline{#1}}

\newcommand{\bu}{\bullet}
\newcommand{\bR}{\bb{R}}

\newcommand{\mS}{\mathbb{S}}
\newcommand{\mSb}{\overline{\mathbb{S}}}

\newcommand{\mM}{\mathcal{M}}

\newcommand{\mC}{\mathcal{C}}

\newcommand{\mD}{\mathcal{D}}
\newcommand{\mT}{\mathcal{T}}
\newcommand{\mA}{\mathcal{A}}
\newcommand{\mAo}{\widetilde{\mathcal{A}}}
\newcommand{\mAh}{\widehat{\mathcal{A}}}

\newcommand{\ddz}{\frac{d}{dz}}
\newcommand{\pduu}[1]{\frac{\partial^{#1}}{{\partial u}^{#1}}}
\newcommand{\pdu}[2]{\left.\frac{\partial^{#1}}{{\partial u}^{#1}}#2\right|_{u=1}}

\newcommand{\bM}{\textrm{M}}

\newcommand{\gA}{\textbf{A}}
\newcommand{\gAo}{\widetilde{\textbf{A}}}
\newcommand{\gAh}{\widehat{\textbf{A}}}
\newcommand{\gB}{\textbf{B}}

\newcommand{\gBh}{\widehat{\textbf{B}}}
\newcommand{\gC}{\textbf{C}}
\newcommand{\gD}{\textbf{D}}
\newcommand{\gM}{\textbf{M}}
\newcommand{\gT}{\textbf{T}}
\newcommand{\gTo}{\widetilde{\textbf{T}}}
\newcommand{\gTh}{\widehat{\textbf{T}}}

\newcommand{\gF}{\textbf{F}}

\newcommand{\gFh}{\widehat{\textbf{F}}}
\newcommand{\gG}{\textbf{G}}

\newcommand{\gY}{\textbf{Y}}
\newcommand{\gYo}{\widetilde{\textbf{Y}}}
\newcommand{\gYh}{\widehat{\textbf{Y}}}
\newcommand{\gW}{\textbf{W}}
\newcommand{\gZ}{\textbf{Z}}

\newcommand{\gN}{\textbf{O}}
\newcommand{\mN}{\mathcal{O}}
\newcommand{\mNh}{\mathcal{P}}
\newcommand{\gNh}{\textbf{P}}
\newcommand{\gNt}{\textbf{Q}}

\newcommand{\NN}{\mathbb{N}}
\newcommand{\DD}{\mathbb{D}}
\newcommand{\RR}{\mathbb{R}}
\newcommand{\CC}{\mathbb{C}}

\newcommand{\TT}{\mathbb{T}}
\newcommand{\TTh}{\mathbb{P}}

\newcommand{\prob}[1]{\mathbb{P}\left[#1\right]}
\newcommand{\ind}{\mathbb{I}_{[0,\infty[}(t)}
\newcommand{\Ee}[1]{\mathbb{E}[#1]}
\newcommand{\EE}[1]{\mathbb{E}\!\left[#1\right]}
\newcommand{\td}{\stackrel{d}{\rightarrow}}

\newcommand{\rX}{\textbf{X}}

\newcommand{\rU}{\textbf{U}}

\title[Simplicial decompositions of surfaces]{Enumerating simplicial decompositions of surfaces with boundaries}

\author[O. Bernardi]{Olivier Bernardi}
\address{O. Bernardi:  CNRS, Département de Mathématiques, Université Paris-Sud, 91405 Orsay, France}
\email{olivier.bernardi@math.u-psud.fr}

\author[J. Rué]{Juanjo Rué}
\address{J. Rué: Departament de Matemàtica Aplicada 2, Universitat Politècnica de Catalunya, 08034 Barcelona, Spain}
\email{juan.jose.rue@upc.edu}

\date{\today}

\thanks{The first author is supported by the French ``Agence Nationale
de la Recherche'', project SADA ANR-05-BLAN-0372. The second author
is supported by the Spanish ``Ministerio de Ciencia e
Innovaci\'on'', project MTM2005-08618-C02-01. This work was
partially supported by the Centre de Recerca matemàtica, Spain.}

\begin{document}

\maketitle

\begin{abstract}
It is well-known that the triangulations of the disc with $n+2$ vertices on its boundary are counted by the $n$th Catalan number $C(n)=\frac{1}{n+1}{2n \choose n}$. This paper deals with the generalisation of this problem to any compact surface $\mS$  with boundaries. We obtain the asymptotic number of simplicial decompositions of the surface $\mS$ with $n$ vertices on its boundary. More generally, we determine the asymptotic number of dissections of $\mS$ when the faces are $\de$-gons with $\de$ belonging to a set of admissible degrees $\De\subseteq \{3,4,5,\ldots\}$. We also give the limit laws for certain parameters of such dissections.
\end{abstract}

\section{Introduction}
It is well-known that the triangulations of the disc with $n+2$ vertices on its boundary are counted by
 the $n$th \emph{Catalan number} $C(n)=\frac{1}{n+1}{2n \choose n}$. This paper deals with the
  generalisation of this problem to any compact surface $\mS$ with boundaries. In particular,
   we will be interested in the asymptotic number of simplicial decompositions of the
    surface $\mS$ having $n$ vertices, all of them lying on the boundary. Some simplicial
decompositions of the disc, cylinder and Moebius band are represented in Figure~\ref{fig:examples-simplicial}.\\

\begin{figure}[ht!]\begin{center} \input{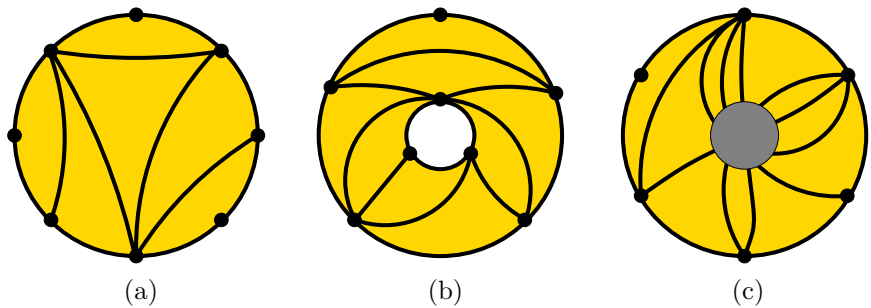}\caption{Simplicial decomposition of (a) the disc, (b) the cylinder, (c) the Moebius band, which is the surface obtained by adding a cross-cap (dashed region) to the disc (points around the cross-cap are identified with their diametral opposite).}\label{fig:examples-simplicial} \end{center}
\end{figure}

In this paper, \emph{surfaces} are \emph{connected} and \emph{compact} 2-dimensional manifolds.
A \emph{map} on a surface $\mS$  is a decomposition of $\mS$ into
 a finite number of 0-cells or \emph{vertices}, 1-cells or \emph{edges} and 2-cells or
\emph{faces}. Maps are considered up to \emph{cell-preserving
homeomorphisms} of the underlying surface (homeomorphisms preserving
each cell of the map). The number of vertices, edges, faces of a map
$M$ on $\mS$ are denoted $v(M)$, $e(M)$ and $f(M)$ respectively.
The quantity $v(M)-e(M)+f(M)$, which depends only on the surface
$\mS$, is called the \emph{Euler characteristic} of the surface and
is denoted $\chi(\mS)$. The \emph{degree} of a face is the number of
incident edges \emph{counted with multiplicity} (an edge is counted
twice if both sides are incident to the face). A map is
\emph{triangular} if every face has degree~3. More generally, given
a set $\De \subset \{1,2,3,\ldots\}$, a map is $\De$-angular if the
degree of any face belongs to $\De$. A map is a \emph{dissection} if
any face of degree $k$ is incident with $k$ distinct vertices and
the intersection of any two faces is either empty, a vertex or an
edge. It is easy to see that triangular maps are dissections if and
only if they have neither loops nor multiple edges.
These maps are called \emph{simplicial decompositions} of $\mS$.\\

In this paper we enumerate asymptotically the simplicial decompositions of an arbitrary surface $\mS$ with boundaries. More precisely, we shall consider the set $\mD_\mS(n)$ of rooted simplicial decomposition of $\mS$ having $n$ vertices, all of them lying on the boundary and prove the asymptotic behaviour
\begin{equation}\label{eq:intro-simplicial}
|\mD_\mS(n)|~\sim_{n\to \infty}~c(\mS)~n^{-3\chi(\mS)/2}~4^n,
\end{equation}
where $c(\mS)$ is a constant which can be determined explicitly. For
instance, the disc $\DD$ has Euler characteristic  $\chi(\DD)=1$ and
the number of simplicial decompositions is
$|\mD_\DD(n)|=C(n\!-\!2)\sim
\frac{1}{\sqrt{\pi}}n^{-3/2}4^n$.\\

We call \emph{non-structuring} the edges which either belong to the boundary of $\mS$ or separate the surface into two parts, one of which is isomorphic to the disc (the other being isomorphic to~$\mS$); the other edges (in particular those that join distinct boundaries) are called \emph{structuring}. We determine the limit laws for the number of structuring edges in simplicial decompositions. In particular, we show that the (random) number $\rU(\mD_\mS,n)$ of structuring edges in a uniformly random simplicial decomposition of a surface $\mS$ with $n$ vertices, rescaled by a factor $n^{-1/2}$, converges in distribution toward a continuous random variable which only depends on the Euler characteristic of $\mS$. \\

We also generalise the enumeration and limit law results to
$\De$-angular dissections for any set of degrees
$\De\subseteq\{3,4,5,\ldots\}$. Our results are obtained by exploiting
a decomposition of the maps in $\mD_\mS(n)$ which is reminiscent of
Wright's work on graphs with fixed excess
\cite{Wright:graph-fixed-excess,Wright:graph-fixed-excess-2} or,
more recently, of work by Chapuy, Marcus and Schaeffer on the
enumeration of unicellular maps
\cite{Chapuy:bijection-quadrangulations-higher-genus}. This decomposition easily translates into an equation satisfied by the corresponding generating function. We then apply classical enumeration techniques based on the analysis of the generating function singularities \cite{Flajol:Sing1,FlajoletSedgewig:analytic-combinatorics}. Limit laws results are obtained by applying the so-called \emph{method of moments} (see \cite{Billingsley:probability}).\\

This paper recovers and extends the asymptotic enumeration and limit
law results obtained via a recursive approach  for the cylinder and
Moebius band in \cite{GaoXiaoWang:triangulations-cylinder},
\cite{Noy:dissection-projective} and \cite{Rue:dissection-cylinder}.
As in these papers, we will be dealing with maps having all their
vertices on the boundary of the surface. This is a sharp restriction
which contrasts with most papers in map enumeration.
However, a remarkable feature of the asymptotic result~\Ref{eq:intro-simplicial} (and the generalisation we obtain for arbitrary set of degrees $\De\subseteq\{3,4,5,\ldots\}$) is  the linear dependency of the polynomial growth exponent in the Euler characteristic of the underlying surface. Similar results were obtained by a recursive method for general maps by Bender and Canfield in \cite{Bender:maps-orientable-surfaces} and for maps with certain degree constraints by Gao in \cite{Gao:degree-restricted-map-general-surface}. This feature as also been re-derived for general maps using a bijective approach in \cite{Chapuy:bijection-quadrangulations-higher-genus}. \\

The outline of the paper is as follows. In Section
\ref{section:definitions}, we recall some definitions about maps and
set our notations. We enumerate rooted triangular maps in
Section~\ref{section:triangular} and then extend the results to
$\De$-angular maps for a general set $\De\subseteq\{3,4,5,\ldots\}$
in Section~\ref{section:D-angular}. It is proved that the number of
rooted $\De$-angular maps with $n$ vertices behaves asymptotically
as $c(\mS,\De) n^{-3\chi(\mS)/2} {\rho_\De}^n$, where $c(\mS,\De)$ and
$\rho_\De$ are constants. In Section~\ref{section:maps-to-dissections}, we prove that the number
of $\De$-angular maps  and of $\De$-angular dissection with $n$
vertices on $\mS$ are asymptotically equivalent as $n$ goes to infinity.
Lastly, in Section~\ref{section:limit-laws}, we study the limit laws of the (rescaled) number of structuring edges in uniformly random $\De$-angular dissections of size $n$. In the Appendix, we give a method for determining the constants $c(\mS,\De)$ explicitly.\\

\section{Definitions and notations}\label{section:definitions}
We denote $\NN=\{0,1,2,\ldots\}$ and $\NN^{\geq k}=\{k,k+1,k+2,\ldots\}$. For any set $\De\subseteq \NN$, we denote by $\gcd(\De)$ the greatest common divisor of $\De$. For any power series $\gF(z)=\sum_{n\geq 0}f_nz^n$, we denote by $[z^n]\gF(z)$ the coefficient $f_n$. We also write $\gF(z)\leq \gG(z)$ if  $[z^n]\gF(z)\leq [z^n]\gG(z)$ for all $n\in\NN$.\\

\noindent \textbf{Surfaces.} \\
Our reference for surfaces is \cite{Mohar:graphs-on-surfaces}. Our \emph{surfaces} are compact, connected and their boundary is homeomorphic to a finite set of disjoint circles. By the \emph{Classification of surface Theorem},
such a surface $\mS$ is  determined, up to homeomorphism, by their Euler characteristic $\chi(\mS)$, the number $\be(\mS)$ of connected components of their boundary and by whether or not they are orientable. 
The orientable surfaces without boundaries are obtained by adding $g\geq 0$
\emph{handles} to the sphere (hence obtaining the $g$-torus with
Euler characteristic $\chi=2-2g$) while non-orientable surfaces without boundaries are
obtained by adding $k>0$ \emph{cross-caps} to the sphere (hence
obtaining the non-orientable surface with Euler characteristic
$\chi=2-k$). For a surface $\mS$ with boundaries, we denote by $\mSb$ the
surface (without boundary) obtained from $\mS$ by gluing a disc on
each of the $\be(\mS)$ boundaries. Observe that the
Euler characteristic $\chi(\mSb)$ is equal to $\chi(\mS)+\be(\mS)$
(since a map on $\mS$ gives rise to a map on $\mSb$ simply by
gluing a face along each of the $\be(\mS)$ boundaries of $\mS$).\\

\noindent \textbf{Maps, rooting and duality.}\\
The \emph{degree} of a vertex in a map is the number of incident edges counted with multiplicity (loops are counted twice). A vertex of degree 1 is called a \emph{leaf}. Given a set $\Delta\subset \NN^{\geq 1}$, a map is called $\Delta$-valent if the degree of every vertex belongs to $\Delta$.\\

An edge of a map has two \emph{ends} (incidence with a vertex)  and
either one or two \emph{sides} (incidence with a face) depending on
whether the edge belongs to the boundary of the surface. A map is
\emph{rooted} if an end and a side of an edge are distinguished as
the \emph{root-end} and \emph{root-side} respectively\footnote{The
rooting of maps on orientable surfaces usually omits the choice of a
root-side because the underlying surface is \emph{oriented} and maps
are considered up to \emph{orientation preserving} homeomorphism.
Our choice of a root-side is equivalent in the orientable case to the choice of an orientation of the surface.}. 
The vertex, edge and face defining these incidences are the \emph{root-vertex},
\emph{root-edge} and \emph{root-face}, respectively.
Rooted maps are considered up to homeomorphism preserving the root-end and -side. In figures, the root-edge will be indicated as an oriented edge pointing away from the root-end and crossed by an arrow pointing toward the root-side. A map is \emph{boundary-rooted} is the root-edge belongs to the boundary of the underlying surface; it is \emph{leaf-rooted} if the root-vertex is a leaf. \\

The \emph{dual} $M^*$ of a map $M$ on a surface without boundary is
a map obtained by drawing a vertex of $M^*$ in each face of $M$ and
an edge of $M^*$ across each edge of $M$. If the map $M$ is rooted,
the root-edge of $M^*$ corresponds to the root-edge $e$ of $M$; the
root-end and root-side of $M^*$ correspond respectively to the side
and end of $e$ which are \emph{not} the root-side and root-end of $M$.
We consider now a (rooted) map $M$ on a surface $\mS$ with boundary. Observe that the (rooted) map $M$  gives raise to a (rooted) map $\bb{M}$ on $\mSb$ by gluing a disc (which become a face of $\overline{M}$) along each components of the boundary of $\mS$. We call \emph{external} these faces of $\bb{M}$ and the corresponding vertices of the dual map $\bb{M}^*$. The \emph{dual} of a map $M$ on a surface $\mS$ with boundary is the map  on $\mSb$ denoted $M^*$ which is obtained from $\overline{M}^*$ by splitting each external vertex of $\overline{M}^*$ of degree $k$, by $k$ special vertices called \emph{dangling leaves}. An example is given in Figure~\ref{fig:duality}. 
Observe that for any set $\De\subseteq\NN^{\geq 2}$, duality establishes a bijection between  boundary-rooted $\Delta$-angular maps on $\mS$ surface $\mS$ and leaf-rooted $\Dee$-valent maps on $\mSb$ having $\be(\mS)$ faces, each of them being incident to at least one leaf.\\

\begin{figure}[ht!]\begin{center} \input{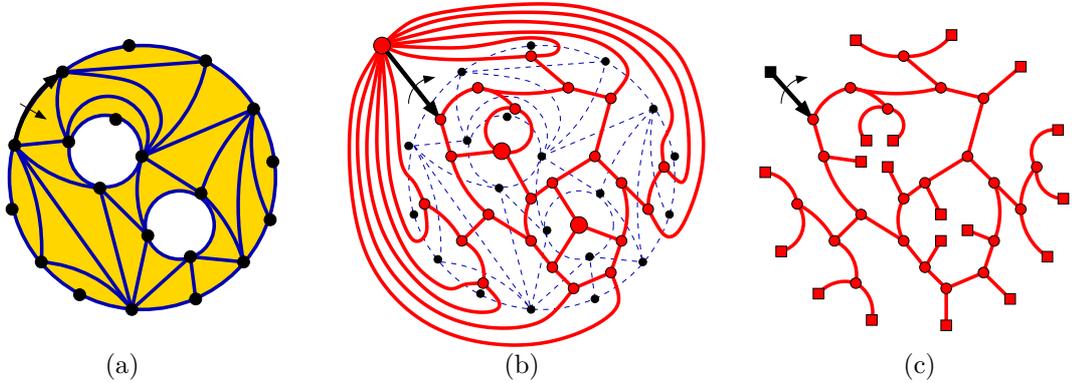}\caption{(a) A rooted map $M$ on the surface obtained by removing 3 disjoint discs to the sphere. (b) The map $\bb{M}^*$ on the sphere. (c) The dual map $M^*$ on the sphere.}\label{fig:duality} \end{center}
\end{figure}

\noindent \textbf{Sets of maps and generating functions.}\\
A \emph{plane tree} is a map on the sphere having a single face. For any set
$\De\subseteq \NN^{\geq 3}$, we denote by  $\mT_\De(n)$ the (finite)
set of $(\Delta\cup\{1\})$-valent leaf-rooted plane trees with $n$
non-root leaves. We denote by $T_\De(n)$ the cardinality of
$\mT_\De(n)$ and by $\gT_\De(z)=\sum_{n\geq 0}T_\De(n)z^n$ the
corresponding generating function.  For the special case of
$\De=\{3\}$, the subscript $\Delta$ will be omitted so that $\mT(n)$
is the set of leaf-rooted \emph{binary trees} with $n$ non-root
leaves. Hence, $T(n+1)=C(n)$ is the $n$th Catalan number
$\frac{1}{n+1}{2n \choose n}$ and $\gT(z)=\sum_{n\geq
0}T(n)z^n=\frac{1-\sqrt{1-4z}}{2}$. A \emph{doubly-rooted} tree is a
leaf-rooted trees having a \emph{marked} leaf distinct from the
root-vertex. Observe that the generating function of
$(\Delta\cup\{1\})$-valent doubly-rooted trees counted by number of
non-root non-marked leaves is $\gT_\De'(z)=\sum_{n\geq 0}(n+1)T_\De(n+1)z^n$. \\

Let $\mS$ be a surface with boundary. For any set $\De\subset
\NN^{\geq 3}$, we denote by $\mM^\De_\mS(n)$  the set of  boundary-rooted
$\De$-angular maps on the surface $\mS$ having $n$ vertices, all of
them lying on the boundary. It is easy to see  that  the number of
edges of maps in $\mM^\De_\mS(n)$ is bounded, hence the set
$\mM^\De_\mS(n)$ is finite. Indeed, for
any map $M$ in $\mM^\De_\mS$ the \emph{Euler relation}  gives
$n-e(M)+f(M)=\chi(\mS)$, the \emph{relation of incidence between
faces and edges} gives $3f(M)+n \leq 2e(M)$, and solving for $e(M)$
gives $e(M)\leq 2n-3\chi(\mS)$. We define $\mD^\De_\mS(n)$ as the subset of (boundary-rooted)
maps in  $\mM^\De_\mS(n)$ which are dissections. We write $\mM^\De_\mS=\cup_{n\geq 1}\mM^\De_\mS(n)$  for the set of all $\De$-angular maps, $\bM^\De_\mS(n)=|\mM^\De_\mS(n)|$ for the number of them having $n$ vertices and $\gM^\De_\mS(z)=\sum_{n\geq 1} \bM^\De_\mS(n)z^n$ for the corresponding generating function. We adopt similar conventions for the set  $\mD^\De_\mS(n)$. Lastly, the subscript $\Delta$ will be omitted in all these notations whenever $\Delta=\{3\}$. For instance, $\gD_\mS(z)$ is the generating function of boundary-rooted simplicial decompositions of the surface $\mS$.\\


\section{Enumeration of triangular maps}\label{section:triangular}
In this section, we consider triangular maps on an arbitrary surface $\mS$ with boundary. We shall enumerate the maps in $\mM_\mS\equiv \mM^{\{3\}}_\mS$ by exploiting a decomposition of the dual $\{1,3\}$-valent maps on $\mSb$. More precisely, we  define a decomposition for maps in the  set $\mA_\mS$ of leaf-rooted $\{1,3\}$-valent maps on $\mSb$ having $\be(\mS)$ faces. Recall that, by duality, the triangular maps in $\mM_\mS$ are in bijection with the maps of $\mA_\mS$ such that each face is incident to at least one leaf (see Figure~\ref{fig:duality}).\\

We denote by $\mA_\mS(n)$ the set of maps in $\mA_\mS$ having $n$ leaves (including the root-vertex). 
If the surface $\mS$ is the disc $\DD$, then $\mA_\DD(n)$ is the set of leaf-rooted binary trees having  $n$ leaves and $|\mM_\DD(n)|=|\mA_\DD(n)|=C(n-2)$, where $C(n)=\frac{1}{n+1}{2n \choose n}$. We now suppose that $\mS$ is not the disc (in particular, the Euler characteristic $\chi(\mS)$ is non-positive). We call \emph{cubic scheme} of the surface $\mS$, a leaf-rooted map on $\mSb$ with $\be(\mS)$ faces, such that that every non-root vertex has degree 3. Observe that one obtains a cubic scheme of the surface $\mS$ by starting from a map in $\mA_\mS$, deleting recursively the non-root vertices of degree 1 and then \emph{contracting} vertices of degree 2 (replacing the two incident edges by a single edge). This process is represented in Figure~\ref{fig:decomposition-cubic}. \\

\begin{figure}[ht!]\begin{center} \input{decomposition-cubic.pstex_t}\caption{The cubic scheme of a map in $\mA_\mS$.}\label{fig:decomposition-cubic} \end{center}
\end{figure}

The cubic scheme obtained from a map $A\in \mA_\mS$ (which is clearly independent of the order of deletions of leaves and of contractions of vertices of degree 2) is called the \emph{scheme} of $A$. The vertices of the scheme $S$ can be
identified with some vertices of $A$. Splitting these vertices gives a
set of doubly-rooted trees, each of them associated to an edge of
the scheme $S$ (see Figure~\ref{fig:definition-Phi}). To be more
precise, let us choose arbitrarily a canonical end and side for each
edge of $S$. Now, the map $A$ is obtained by replacing each edge $e$
of $S$ by a doubly-rooted binary tree $\tau_e^\bu$ in such a way
that the canonical end and side of the edge $e$ coincide with the
root-end and root-side of the tree $\tau_e^\bu$. It is easy to see
that any cubic scheme of the surface $\mS$ has $2-3\chi(\mS)$ edges
(by using Euler relation together with the relation of incidence
between edges and vertices).  Therefore, upon choosing an arbitrary
labelling and canonical end and side for the edges of every scheme
of $\mS$, one can define a mapping $\Phi$ on $\mA_\mS$ by setting
$\Phi(A)=(S,(\tau_1^\bu,\ldots,\tau_x^\bu))$, where $S$ is the
scheme of the map $A$ and $\tau_i^\bu$ is the doubly-rooted tree
associated with the $i$th edge of $S$. Reciprocally, any pair
$(S,(\tau_1^\bu,\ldots,\tau_x^\bu))$ defines a map in $\mA_\mS$,
hence the following lemma.

\begin{lemma}\label{lem:decomposition-cubic}
The mapping $\Phi$ is a bijection between
\begin{itemize}
\item the set $\mA_\mS$ of leaf-rooted $\{1,3\}$-valent maps on $\mSb$ having $\be(\mS)$ faces,
\item the pairs made of a cubic scheme of $\mS$ and a sequence of $2-3\chi(\mS)$ doubly-rooted binary trees.
\end{itemize}
Moreover, the number of non-root leaves of a map $A\in \mA_\mS$ is
equal to the total number of leaves which are neither marked- nor root-leaves in the associated sequence of doubly-rooted trees.
\end{lemma}

\begin{figure}[ht!]\begin{center} \input{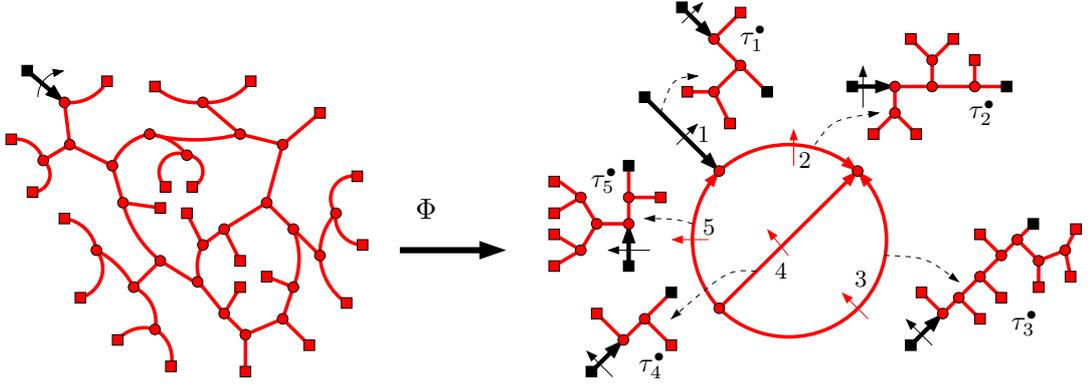}\caption{The bijection $\Phi$. Root-leaves and marked leaves are indicated by black squares.}\label{fig:definition-Phi} \end{center}
\end{figure}

We  now exploit Lemma~\ref{lem:decomposition-cubic} in order to prove the following Theorem.

\begin{thm}\label{thm:nb-triangular}
Let $\mS$ be any surface with boundary distinct from the disc $\DD$. The asymptotic number of boundary-rooted triangulations on $\mS$ with $n$ vertices, all of them lying on the boundary is
\begin{equation}\label{eq:asympt-simplicial}
M_\mS(n)~=_{n\to \infty}~~
\frac{a(\mS)}{4\Gamma(1\!-\!3\chi(\mS)/2)}~ n^{-3\chi(\mS)/2} ~4^n\, \left(1+O\left(n^{-1/2}\right)\right),
\end{equation}
where $a(\mS)$ is the number of cubic schemes of $\mS$ and $\Gamma$ is the usual Gamma function.
\end{thm}

A way of determining the constant $a(\mS)$ is given in the Appendix~\ref{section:appendix-constants}. The rest of this section is devoted to the proof of Theorem~\ref{thm:nb-triangular}.\\

Lemma~\ref{lem:decomposition-cubic} immediately translates into an equation relating the generating function
 $\gA_\mS(z)\equiv\sum_{n\geq 0} |\mA_\mS(n)|z^n$ and the generating function $\gT'(z)$ of doubly-rooted binary trees:
$$ 
\gA_\mS(z)~=~ a(\mS)z\left(\gT'(z)\right)^{2-3\chi(\mS)},
$$ 
where $a(\mS)$ is the number of cubic schemes of the surface $\mS$.
Since the cubic schemes of $\mS$ are trivially in bijection with
rooted $\{3\}$-valent maps on $\mSb$ having $\be(\mS)$ faces, the
constant $a(\mS)$ counts these maps. Moreover, since
$\gT(z)=\sum_{n\geq 0}C(n)z^{n+1}=(1-\sqrt{1-4z})/2$ one gets
$\gT'(z)=(1-4z)^{-1/2}$ and
$$\gA_\mS(z)= a(\mS)z(1-4z)^{3\chi(\mS)/2-1}.$$
From this expression, one can obtain an exact formula (depending on the parity of $\chi(\mS)$) for the cardinal $|\mA_\mS(n)|=[z^n]\gA_\mS(z)$. However, we shall content ourselves with the following asymptotic result:
\begin{equation}\label{eq:asymptotic-A}
|\mA_\mS(n)|\equiv[z^n]\gA_\mS(z)~=_{n\to \infty}~\frac{a(\mS)}{4\Gamma\left(1-3\chi(\mS)/2\right)}~n^{-3\chi(\mS)/2}\, 4^n~ \left(1+O\left(n^{-1}\right)\right).
\end{equation}

It only remains to bound the number of maps in $\mA_\mS(n)$ which are not the dual of maps in $\mM_\mS(n)$ to prove the following lemma.
\begin{lemma}\label{lem:comparison-AandM}
The number $M_\mS(n)$  of boundary-rooted  triangular maps on $\mS$ having $n$ vertices satisfies:
$$M_\mS(n)=A_\mS(n)\left(1+O(n^{-1/2})\right).$$
\end{lemma}

A doubly-rooted tree is said \emph{one-sided} if there are no-leaf on one of the sides of the path going from the root-vertex to the marked leaf; it is said \emph{two-sided} otherwise.
 
\begin{proof} Let $\mAo_\mS$ be the class of maps in $\mA_\mS$ which are not
the dual of maps in $\mM_\mS$ and let
$\gAo_\mS(z)=\gA_\mS(z)-\gM_\mS(z)$ be the corresponding generating
function. Let $A$ be a map in $\mA_\mS$ which is not the dual of a
triangular map in $\mM_\mS$. Then $A$ has a face incident to no leaf
and its image $(S,(\tau_1^\bu,\ldots,\tau_{2-3\chi(\mS)}^\bu))$ by
the bijection $\Phi$ is such that one of the doubly-rooted trees
$\tau_1^\bu,\ldots,\tau_{2-3\chi(\mS)}^\bu$ is one-sided. Thus,
$$\gAo_\mS(z)\leq \left(2-3\chi(\mS)\right)\,a(\mS) z \gTo(z) (\gT'(z))^{1-3\chi(\mS)},$$
where  $\gTo(z)$ is the generating function of one-sided
doubly-rooted binary trees (counted by number of non-root,
non-marked leaves). The number of one-sided doubly-rooted binary
trees having $n$ leaves which are neither marked nor the root-vertex
is $2~T(n+1)$ if $n>0$ and 1 for $n=0$, hence $\gTo(z)=2\gT(z)/z-1$.
The coefficients of the series $z\gTo(z) (\gT'(z))^{1-3\chi(\mS)}$
can be determined explicitly and gives
$$\displaystyle [z^n]\gAo_\mS(z)=O\left([z^n]z\gTo(z)(\gT'(z))^{1-3\chi(\mS)}\right)=O\left(n^{-3\chi(\mS)/2-1/2}4^n\right)=O\left(\frac{[z^n]\gA_\mS(z)}{\sqrt{n}}\right).$$
This completes the proof of  Lemma~\ref{lem:comparison-AandM}. 
\end{proof}

Lemma~\ref{lem:comparison-AandM} together with Equation~\Ref{eq:asymptotic-A} complete the proof of Theorem~\ref{thm:nb-triangular}.
\findem


\section{Enumeration of $\De$-angular maps}\label{section:D-angular}
We  now extend the results of the previous section to $\De$-angular
maps, where $\De$ is any subset of $\NN^{\geq 3}$. We first deal
with the case of the disc $\DD$. By duality, the problem corresponds
to counting leaf-rooted $\Dee$-valent trees by number of leaves.
This is partially done in \cite[Example VII.13]{FlajoletSedgewig:analytic-combinatorics} and we follow the
method developed there. Then, we count $\De$-angular maps on
arbitrary surfaces by exploiting an extension of the bijection $\Phi$.

\subsection{Counting trees by number of leaves}
In this subsection, we enumerate the $\Dee$-valent plane trees by number of leaves. 

\begin{prop}\label{prop:nb-Dary-trees-general}
Let $\De$ be any subset of $\NN^{\geq 3}$ and let $p$ be the greatest common divisor of $\{\de-2,~\de\in\De\}$.
Then, the number of non-root leaves of $\Dee$-valent trees are congruent to 1 modulo $p$. Moreover, the asymptotic number of leaf-rooted $\Dee$-valent trees having $np+1$ non-root leaves is
$$T_\De(np+1)=_{n\to \infty}\frac{p\ga_\De}{2\sqrt{\pi}}~ (np+1)^{-3/2}\, {\rho_\De}^{-(np+1)}\,\left(1+O\left(n^{-1}\right)\right),$$
where $\tau_\De$, $\rho_\De$ and $\ga_\De$ are the unique positive constants satisfying:
\begin{equation}\label{eq:constants-trees}
\displaystyle \sum_{\de\in\De}(\de-1) {\tau_\De}^{\de-2}=1, ~~\rho_\De=\tau_\De-\sum_{\de\in\De}{\tau_\De}^{\de-1} ~\textrm{ and }~ \displaystyle \ga_\De=\sqrt{\frac{2\rho_\De}{\sum_{\de\in\De}(\de-1)(\de-2) {\tau_\De}^{\de-3}}}.
\end{equation}
\end{prop}

\noindent \textbf{Remarks on the constants $\tau_\De,~\rho_\De$ and $\ga_\De$:}
\ite The positive constant $\tau_\De$ satisfying
\Ref{eq:constants-trees} clearly exists, is unique, and is less than 1. Indeed, the
function $f:\tau \mapsto\sum_{\de\in\De}(\de-1) \tau^{\de-2}$ is well-defined and 
strictly increasing on $[0,1[$ and $f(0)=1$ while $\lim_{\tau \to 1} f(\tau)>1$.
Moreover,
$\rho_\De=\tau_\De\left(1-\sum_{\de\in\De}{\tau_\De}^{\de-2}\right)$
is clearly positive. Hence, $\tau_\De,~\rho_\De$ and $\ga_\De$
are well defined.
\ite An important case is when $\De$ is made of a single element $\de=p+2$.
In this case, one gets $\tau_{\De}=(p+1)^{-1/p}$, $\rho_{\De}=\left(\frac{p^p}{(p+1)^{p+1}}\right)^{1/p}$, and
 $\ga_{\De}=\sqrt{2(p+1)^{-(p+2)/p}}$. \\ 

We first introduce a change of variable in order to deal with the periodicity of the number of leaves.

\begin{lemma}\label{lem:periodicity}
Let $\De\subseteq \NN^{\geq 3}$, let $p=\gcd(\de-2,~\de\in\De)$ and let $\gT_\De(z)$ be the generating function of $\Dee$-valent
leaf-rooted trees. There exists a unique power series
$\gY_\De(t)$ in $t$ such that
$$\gT_\De(z)=z\gY_\De(z^p).$$
Moreover, the series $\gY_\De$ satisfies
\begin{equation}\label{eq:decomposition-tree-periodic}
\gY_\De(t)=1+\sum_{k\in K}t^{k}\gY_\De(t)^{kp+1}
\end{equation}
where $K$ is the subset of $\NN^{\geq 1}$ defined by $\De=\{2+kp,~k\in K\}$.
\end{lemma}

\begin{proof} The fact that the number of non-root leaves is congruent to 1
modulo $p$ is easily shown by induction on the number of leaves.  Hence a 
power series $\gY_\De(t)$ such that $\gT_\De(z)=z\gY_\De(z^p)$ exists and is
unique with this property. We now use the classical \emph{decomposition of trees at the root} (which corresponds to splitting the vertex adjacent to the root-leaf) represented in Figure~\ref{fig:decomposition-tree}. This decomposition gives 
\begin{equation}\label{eq:decomposition-tree}
\gT_\De(z)=z+\sum_{\de\in \De}\gT_\De(z)^{\de-1},
\end{equation}
and one obtains~\Ref{eq:decomposition-tree-periodic} by substituting $\gT_\De(z)$ by $z\gY_\De(z^p)$ in~\Ref{eq:decomposition-tree} and then substituting $z^p$ by $t$.
\end{proof}

\begin{figure}[ht!]\begin{center} \input{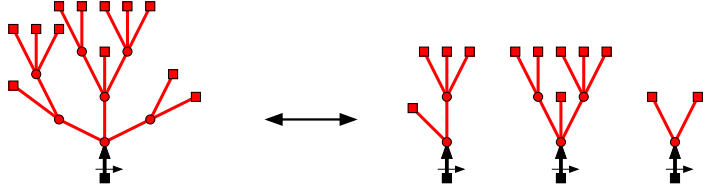}\caption{Decomposition of trees at the root}\label{fig:decomposition-tree} \end{center}
\end{figure}

We now analyse the singularity of the generating function
$\gY_\De(t)$ and deduce from it the asymptotic behaviour of its coefficients.
We call \emph{domain dented} at a value $R>0$, a domain of the
complex plane $\CC$ of the form $\{z\in \CC: |z|<R' \textrm{ and }
\arg(z-R)\notin [-\theta,\theta]\}$ for some real number $R'>R$ and
some positive angle $0<\theta<\pi/2$. A dented domain is represented
in Figure~\ref{fig:dented-domain}.

\begin{lemma}\label{lem:singularities-trees-general}
Let $\De\subseteq \NN^{\geq 3}$ and let $p$, $\tau_\De,~\rho_\De$ and $\ga_\De$ be as defined in Proposition~\ref{prop:nb-Dary-trees-general}. The generating function $\gY_\De(t)$ of leaf-rooted $\Dee$-valent trees (defined in Lemma~\ref{lem:periodicity})  is analytic in a domain dented at  $t={\rho_\De}^p$. Moreover, at any order $n\geq 0$, an expansion of the form 
$$\gY_\De(t)~=_{t \to {\rho_\De}^p}~\sum_{k=0}^n \al_k\left(1-\frac{t }{{\rho_\De}^p}\right)^{k/2}+o\left(\left(1-\frac{t}{{\rho_\De}^p}\right)^{n/2}\right)$$ 
is valid in this domain, with $\displaystyle \alpha_0=\frac{\tau_\De}{\rho_\De}$ and $\displaystyle \alpha_1=-\frac{\ga_\De}{\rho_\De\sqrt{p}}$.
\end{lemma}

The proof of Lemma~\ref{lem:singularities-trees-general} uses a theorem Meir and Moon \cite{Meier:asymptotic-method} about generating function defined by a \emph{smooth implicit-function schema}. This proof is given in Appendix~\ref{appendix:asymptotic-method}.\\

Lemma~\ref{lem:singularities-trees-general} ensures that the generating function $\gY_\De(t)$ satisfies the required condition in order to apply the classical \emph{transfer theorem} between singularity types and coefficient asymptotics (see \cite{FlajoletSedgewig:analytic-combinatorics} for details). Applying this transfer theorem gives Proposition~\ref{prop:nb-Dary-trees-general}.\findem


\subsection{Counting $\De$-angular maps on general surfaces}
We consider a surface $\mS$ with boundary distinct from the disc and denote by  $\mA_\mS^\De$ the set of leaf-rooted $\Dee$-valent maps on $\mSb$. Recall that by duality the maps in $\mM^\De_\mS$ (that is, boundary-rooted $\De$-angular on the surface $\mSb$) are in bijection with the maps in $\mA_\mS^\De$ such that every face is incident with at least one leaf.\\

We first extend the bijection $\Phi$ defined in Section~\ref{section:triangular} to maps in $\mA_\mS^\De$. This decomposition leads to consider leaf-rooted trees with \emph{legs}, that is, marked vertices at distance~2 from the root-leaf.  We call \emph{scheme} of the surface $\mS$ a leaf-rooted map on $\mSb$ having $\be(\mS)$ faces and such that the degree of any non-root vertex is at least $3$. Recall that a scheme is said \emph{cubic} if the degree of any non-root vertex is 3. By combining Euler relation with the relation of incidence between vertices and edges, it is easy to see that the number of edges of a scheme of $\mS$ is at most $e=2-3\chi(\mS)$, with equality if and only if the scheme is cubic. In particular, this implies that the number of scheme is finite. \\


Let $A$ be a map in $\mA_\mS^\De$.  One obtains a scheme $S$ of $\mS$ by recursively deleting the non-root
vertices of degree 1 and then contracting the vertices of degree 2;
see Figure~\ref{fig:decomposition-cubic}. The vertices of the scheme
$S$  can be identified with some  vertices of $A$. Splitting these
vertices gives a sequence of doubly-rooted $\De$-valent trees (each
of them associated with an edge of $S$), and a sequence of
leaf-rooted $\De$-valent trees with legs (each of them associated
with a non-root vertex of $S$); see Figure~\ref{fig:extension-Phi}.
More precisely, if the scheme $S$ has  $k$ edges and $l$ non-root
vertices having degree $d_1,\ldots,d_l$, then the map $A$ is
obtained by substituting each edge of $S$ by a doubly-rooted
$\Dee$-valent tree  and each vertex of $S$ of degree $d$ by a
leaf-rooted $\Dee$-valent trees with $d-1$ legs. (These
substitution can be made unambiguously provided that one chooses a
canonical end and side for each each edge of $S$ and a canonical
incident end for each vertex of $S$.) We define the mapping $\Phi$
on $\mA^\De_\mS$ by setting
$\Phi(A)=(S,(\tau_1^\bu,\ldots,\tau_e^\bu),(\tau_1,\ldots,\tau_v))$,
where $S$ is the scheme of $A$ having $e$ edges and $v$ non-root
vertices, $\tau_i^\bu$ is the doubly-rooted tree associated with the
$i$th edge of $S$ and $\tau_j$ is the tree with legs associated with
the $j$th vertex of~$S$. Reciprocally any triple
$(S,(\tau_1^\bu,\ldots,\tau_e^\bu),(\tau_1,\ldots,\tau_v))$ defines
a map in $\mA_\mS^\De$, hence the following lemma.

\begin{lemma}\label{lem:decomposition-D}
Let $S$ be a scheme of $~\mS$ with $e$ edges and $v$ non-root
vertices having degree $d_1,\ldots,d_v$. Then the  mapping $\Phi$
gives a bijection between
\begin{itemize}
\item the set map in $\mA^\De_\mS$ having scheme $S$,
\item pairs made of a sequence of $e$ doubly-rooted $\Dee$-valent trees and a sequence of  $v$  leaf-rooted $\Dee$-valent trees  with $d_1\!-\!1,\ldots,d_v\!-\!1$ legs respectively.
\end{itemize}
Moreover, the number of non-root leaves of a map $A\in \mA_\mS^\De$ is equal to the total number of leaves which are neither legs nor marked-leaves nor root-leaves in the associated sequences of trees.
\end{lemma}

\begin{figure}[ht!]\begin{center} \input{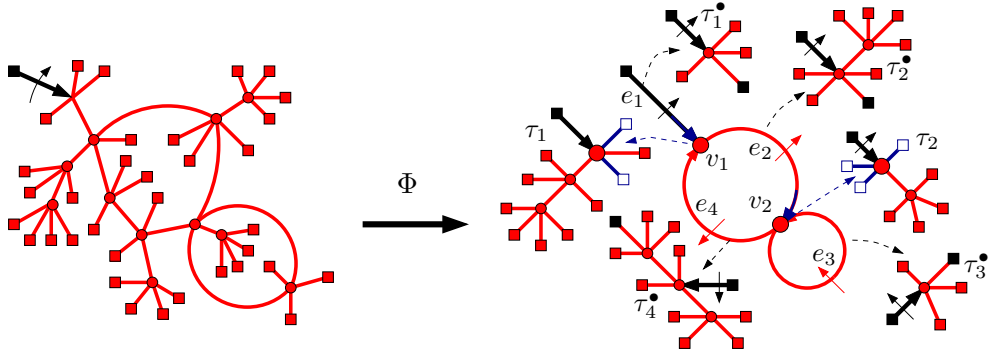}\caption{The bijection $\Phi$: decomposition of a map whose scheme has $e=4$ edges and $v=2$ non-root vertices of respective degrees 3 and 4. Legs  are indicated by white squares.}\label{fig:extension-Phi} \end{center}
\end{figure}

We will now use Lemma~\ref{lem:decomposition-D} in order to prove the following enumeration result.

\begin{thm}\label{thm:nb-Dangular}
Let $\mS$ be any surface with boundary distinct from the disc, let $\De\subseteq\NN^{\geq 3}$ and let $p$ be the greatest common divisor of $\{\de-2,~\de\in\De\}$. Then, the number of vertices of $\De$-angular maps (having all their vertices on the boundary)  is congruent to $2\chi(\mS)$ modulo $p$. Moreover, the asymptotic number of boundary-rooted $\De$-angular maps of $\mS$ having $np+2\chi(\mS)$ vertices is
\begin{equation}\label{eq:asympt-dissection}
M^\De_\mS(np+2\chi(\mS)) =_{n\to
\infty}\frac{a(\mS)p\left(8\ga_\De\rho_\De\right)^{\chi(\mS)}}{4\Gamma(1\!-\!3\chi(\mS)/2)}~(np\!+\!2\chi(\mS))^{-3\chi(\mS)/2}
\,
{\rho_\De}^{-(np+2\chi(\mS))}~\left(1+O\left(n^{-1/2}\right)\right),
\nonumber
\end{equation}
where $\rho_\De$ and $\ga_\De$ are the constants determined by
Equation~\Ref{eq:constants-trees}, and $a(\mS)$ is the number of
cubic schemes of $~\mS$.
\end{thm}

Theorem~\ref{thm:nb-Dangular} generalises Theorem~\ref{thm:nb-triangular} since for $\De=\{3\}$, one has $p=1$, $\rho_\De=1/4$ and $\ga_\De=1/2$. The rest of this section is devoted to the proof of Theorem~\ref{thm:nb-Dangular}. \\

We denote by $\gA_\mS^\De(z)$ the generating function of the set $\mA_\mS^\De$ of leaf-rooted $\Dee$-valent maps counted by number of leaves. By partitioning the set of maps in $\mA_\mS^\De$ according to their scheme,  one gets
\begin{equation}\label{eq:sum-schemes}
\gA_\mS^\De(z)= \sum_{S\textrm{ scheme}}\gF_S^\De(z),
\end{equation}
where the sum is over the schemes of $\mS$ and $\gF_S^\De(z)$ is the generating function of the subset of  maps in $\mA_\mS^\De$ having scheme $S$ counted by number of leaves. Moreover, for a scheme $S$ with  $e$ edges and $v$ non-root vertices of respective degrees $d_1,\ldots,d_v$,  Lemma~\ref{lem:decomposition-D} gives
\begin{equation}\label{eq:product-degree}
\gF_S^\De(z)~=~z\, \left(\gT'_\De(z)\right)^e\, \prod_{i=1}^v \gT_{\De,d_i-1}(z),
\end{equation}
where $\gT_{\De,\ell}(z)$ is the generating function of $\Dee$-valent leaf-rooted trees with $\ell$ legs counted by number of leaves which are neither legs nor root-leaves.\\

It is now convenient to introduce a change of variable for dealing with the periodicity of the number of leaves of maps in $\mA_\mS^\De$.

\begin{lemma}\label{lem:periodicity2}
Let $\De\subseteq \NN^{\geq 3}$ and let $p=\gcd(\de-2,~\de\in\De)$. For any  positive integer $\ell$,
there exists a power series $\gY_{\De,\ell}(t)$ in $t$ such that
$\gT_{\De,\ell}(z)=z^{1-\ell}\gY_{\De,\ell}(z^p)$. Moreover, the
generating function of  leaf-rooted $\Dee$-valent maps on $\mSb$
satisfies
\begin{equation}\label{eq:sum-product}
\gA^\De_\mS(z)=z^{2\chi(\mS)}\gB^\De_\mS(z^p),~ \textrm{ with } \gB^\De_\mS(t)=\!\!\!\sum_{S\textrm{ scheme}}\!\left(pt\gY_{\De}'(t)+\gY_{\De}(t)\right)^{|E_S|}\, \prod_{u \in V_S} \gY_{\De,\deg(u)-1}(t).
\end{equation}
where the sum is over all the schemes $S$ of $\mS$ and $E_S$, $V_S$ are respectively the set of edges and non-root vertices of the scheme $S$ and $\deg(u)$ is the degree of the vertex $u$.
\end{lemma}

Observe that
Equation~\Ref{eq:sum-product} shows that the coefficient of $z^n$ in
the series $\gA_\mS^\De(z)$ is $0$ unless $n$ is congruent to
$2\chi(\mS)$ modulo $p$.
In other words, the number of leaves of $\Dee$-valent maps on $\mSb$  is congruent to $2\chi(\mS)$ modulo $p$. \\

\begin{proof} 
\ite By Lemma~\ref{lem:periodicity}, the number of non-root
leaves of $\Dee$-valent trees  is congruent to~$1$ modulo~$p$.
Hence, the number of non-root, non-marked leaves of a $\Dee$-valent
trees with $\ell$ legs is congruent to $1-\ell$ modulo $p$. This
ensures the existence of the power series $\gY_{\De,\ell}(t)$ such
that $\gT_{\De,\ell}(z)=z^{1-\ell}\gY_{\De,\ell}(z^p)$. \ite Let $S$
be a scheme with  $e$ edges and $v$ non-root vertices of respective
degrees $d_1,\ldots,d_v$. Plugging the identities
$\gT_{\De,\ell}(z)=z^{1-\ell}\gY_{\De,\ell}(z^p)$
into~\Ref{eq:product-degree} gives
\begin{equation}\label{eq:product-degree2}
\gF_S^\De(z)~=~z^{1+2v-\sum_{i=1}^v d_i}\, \left(\ddz\left[z\gY_\De(z^p)\right]\right)^e \, \prod_{i=1}^v \gY_{\De,d_i-1}(z^p).
\end{equation}
The sum $\sum_{i=1}^v d_i$ is the number of edge-ends which are not
the root-ends. Hence, $\sum_{i=1}^v d_i=2e-1$ and by Euler relation,
$1+2v-\sum_{i=1}^v
d_i=2+2v-2e=2\left(\chi(\mSb)-\be(\mS)\right)=2\chi(\mS)$ (since the
scheme $S$ is a map on $\mSb$ having $v+1$ vertices, $e$ edges and
$\be(\mS)$ faces). Moreover,
$\ddz\left[z\gY_\De(z^p)\right]=pz^p\gY'_\De(z^p)+\gY_\De(z^p)$.
Thus continuing Equation~\Ref{eq:product-degree2}, gives
\begin{equation}\label{eq:product-degree3}
\gF_S^\De(z)~=~z^{2\chi(\mS)}\, \left(pz^p\gY_\De'(z^p)+\gY_\De(z^p)\right)^{e}~ \prod_{i=1}^v \gY_{\De,d_i-1}(z^p),
 \end{equation}
Replacing $z^p$ by $t$ in the right-hand-side of~\Ref{eq:product-degree3} and summing over all the schemes of $\mS$ gives Equation~\Ref{eq:sum-product} from Equation~\Ref{eq:sum-schemes}.
\end{proof}

We now study the singularities of the generating functions of trees with legs.

\begin{lemma} \label{lem:singularity}
Let $\De\subseteq \NN^{\geq 3}$, let  $p=\gcd(\de-2,~\de\in\De)$ and let $\rho_\De$ and $\ga_\De$ be the constants defined by Equation~\Ref{eq:constants-trees}. For any positive integer $\ell$, the generating function $\gY_{\De,\ell}(t)$ is analytic in a domain dented at 
$t={\rho_\De}^p$. Moreover, there exists a constant $\ka_\ell$ such that the expansion
$$\gY_{\De,\ell}(t)=_{t\to {\rho_\De}^p} \ka_\ell+O\left(\sqrt{1-\frac{t}{{\rho_\De}^p}}\right),$$
is valid in this domain. In particular, $\displaystyle \ka_2= \left(\frac{\rho_\De}{\ga_\De}\right)^2$.
\end{lemma}

\begin{proof} 
\ite By considering the decomposition at the root of $\Dee$-valent leaf-rooted trees with $\ell$ legs, one gets
\begin{equation}\label{eq:tree+legs}
\gT_{\De,\ell}(z)=\sum_{\de\in \De,~\de> \ell}{\de-1 \choose l}\gT_\De(z)^{\de-\ell-1}.
\end{equation}
In particular, for $\ell=1$ this gives
$$\gT_{\De,1}(z)=\sum_{\de\in\De}(\de-1)\gT_{\De}(z)^{\de-2}~=~\frac{1}{\gT_\De'(z)}\frac{d}{dz}\pare{\sum_{\de\in\De}\gT_{\De}(z)^{\de-1}}$$
And using Equation~\Ref{eq:decomposition-tree} gives
\begin{eqnarray}\label{eq:T1}
\gT_{\De,1}(z)&=&\frac{1}{\gT_\De'(z)}\frac{d}{dz}(\gT_\De(z)-z) ~=~1-\frac{1}{\gT_\De'(z)}.
\end{eqnarray}
Moreover for $\ell>1$, Equation~\Ref{eq:tree+legs} gives 
$$\gT_{\De,\ell}(z)= \frac{1}{\ell\gT'_\De(z)}\frac{d}{dz}\gT_{\De,\ell-1}(z).$$
Making the change of variable $t=z^p$ gives 
\begin{equation}\label{eq:rec-legs}
\gY_{\De,1}(t)=1-\frac{1}{\gZ_\De(t)}~~~\textrm{ and }~~~\forall \ell\geq 2, \gY_{\De,\ell}(t)=\frac{pt\gY_{\De,\ell-1}'(t)+(2-\ell)\gY_{\De,\ell-1}(t)}{\ell\gZ_\De(t)},
\end{equation}
where $\gZ_\De(t)\equiv pt\gY'_\De(t)+\gY_\De(t)$ is such that $\gT'_\De(z)=\gZ_\De(z^p)$. In particular, 
\begin{eqnarray}\label{eq:T2}
\gT_{\De,2}(t)=\frac{pt\gZ_\De'(t)}{2\gZ_\De(t)^3}.
\end{eqnarray}
\ite We now prove that \emph{for all $\ell>0$, the generating function $\gY_{\De,\ell}(t)$ is analytic in a domain dented at  $t={\rho_\De}^p$.}\\
By Lemma~\ref{lem:singularities-trees-general}, the generating function $\gZ_\De(t)\equiv pt\gY'_\De(t)+\gY_\De(t)$  is analytic in a domain dented at  $t={\rho_\De}^p$. Given Equation~\Ref{eq:rec-legs}, the same property holds for the series $\gY_{\De,\ell}(t)$ for all $\ell>0$ provided that $\gZ_\De(t)$ does not cancel in a domain dented at $t={\rho_\De}^p$. It is therefore sufficient to prove that
$\left|\gZ_\De(t)\right|>1/2$ for all $|t|<{\rho_\De}^p$ or equivalently, $\left|\gT'_\De(z)\right|>1/2$ for $|z|<\rho_\De$.

First observe that $\lim_{z\to \rho_\De} \gT_\De(z)= \tau_\De$ by Lemma~\ref{lem:singularities-trees-general}. Hence, for all $|z|<\rho_\De$, $\left|\gT_\De(z)\right|\leq \gT_\De(|z|)< \tau_\De$ and by~\Ref{eq:constants-trees}
$$\left|\sum_{\de\in\De}(\de-1){\gT_\De(z)}^{\de-2}\right|\leq \sum_{\de\in\De}(\de-1)\left|\gT_\De(z)\right|^{\de-2}< \sum_{\de\in\De}(\de-1)\tau_\De^{\de-2}=1.$$ 
Moreover, by differentiating~\Ref{eq:decomposition-tree} with respect to $z$, one gets $$\gT'_\De(z)=1+\gT'_\De(z)\sum_{\de\in\De}(\de-1){\gT_\De(z)}^{\de-2}.$$
Hence, $\displaystyle |\gT'_\De(z)|\geq\frac{1}{1+\left|\sum_{\de\in\De}(\de-1){\gT_\De(z)}^{\de-2}\right|}>\frac{1}{2}$.\\
\ite We now prove that  \emph{for all $\ell>0$, the series $\gY_{\De,\ell}(t)$ has an expansion of the form $\ka_\ell+O\left(\sqrt{1-t/{\rho_\De}^p}\right)$ valid in a domain dented at $t={\rho_\De}^p$.} \\
By Lemma~\ref{lem:singularities-trees-general}, the series $\gY_\De(t)$ is analytic in a domain at $t={\rho_\De}^p$. Thus, its expansion at $t={\rho_\De}^p$ can be differentiated term by term. For the series $\gZ_\De(t)\equiv pt\gY'_\De(t)+\gY_\De(t) $ this gives an expansion of the form 
\begin{equation}\label{eq:expansionX}
\gZ_\De(t)~=_{t\to{\rho_\De}^p}~ \left(1-\frac{t}{{\rho_\De}^p}\right)^{-1/2}\left[\sum_{k=0}^n \be_k  \left(1-\frac{t}{{\rho_\De}^p}\right)^{k/2}+o\left(\left(1-\frac{t}{{\rho_\De}^p}\right)^{n/2}\right)\right],
\end{equation}
where $\be_0=-\frac{p}{2}\al_1=\frac{\ga_\De\sqrt{p}}{2\rho_\De}$. Thus, by induction on $\ell$,  the series $\gY_{\De,\ell}(z)$  has an
expansion of the form 
$$\gY_{\De,\ell}(z)=_{t\to{\rho_\De}^p} \sum_{k=0}^n \kappa_{\ell,k} (1-\frac{t}{{\rho_\De}^p})^{k/2}+o\left((1-\frac{t}{{\rho_\De}^p})^{n/2}\right).$$
In particular, Equation~\Ref{eq:T2} gives $\kappa_2\equiv\kappa_{2,0}=\frac{p}{4{\be_0}^4}= \left(\rho_\De/\ga_\De\right)^2$.
\end{proof}

We now complete the proof of Theorem~\ref{thm:nb-Dangular}. By Equation~\Ref{eq:expansionX},
\begin{equation}\label{eq:asymptW}
pt\gY_\De'(t)+\gY_\De(t)\equiv \gZ_\De(t) =_{t\to {\rho_\De}^p}~ \frac{\ga_\De\sqrt{p}}{2{\rho_\De}} \left(1-\frac{t}{{\rho_\De}^p}\right)^{-1/2}\left(1+O\left(1-\frac{t}{{\rho_\De}^p}\right)\right)
\end{equation}
in a domain dented at $t= {\rho_\De}^p$. Thus, by Lemma
\ref{lem:singularity}, the generating function $\displaystyle
\gG_S(t)\equiv \left(pt\gY_\De'(t)+\gY_\De(t)\right)^{e}\,
\prod_{i=1}^{v} \gY_{\De,d_i - 1}(t)$ associated to a scheme~$S$
with $e$ edges and $v$ vertices of respective degrees
$d_1,\ldots,d_v$ has an expansion of the form
$$\gG_S(t)  =_{t\to {\rho_\De}^p}\left(\prod_{i=1}^{v}\ka_{d_i - 1} \right)
\parfrac{\ga_\De\sqrt{p}}{2{\rho_\De}}^{e}
\left(1-\frac{t}{{\rho_\De}^p}\right)^{-e/2}
\left(1+O\left(\sqrt{1-\frac{t}{{\rho_\De}^p}}\right)\right),$$
valid in a domain dented at $t= {\rho_\De}^p$. Moreover, as mentioned above, the number of edges of a scheme of
$\mS$ is at most $e=2-3\chi(\mS)$, with equality if and only if the
scheme is cubic. 
There are $a(\mS)>0$ cubic schemes and all of them
have $v=1-2\chi(\mS)$ non-root vertices. Thus, Lemma~\ref{lem:periodicity2} gives $\gA^\De_\mS(z)=z^{2\chi(\mS)}\gB^\De_\mS(z^p)$ with 
\begin{eqnarray}
\gB^\De_\mS(t) &&\!\!\!\!\!=  \!\!\sum_{S\textrm{ scheme}}\gG_S(t)\nonumber \\
&&\!\!\!\!\!=_{t\to {\rho_\De}^p} a(\mS) \parfrac{\rho_\De}{\ga_\De}^{2-4\chi(\mS)}\!\!\parfrac{\ga_\De\sqrt{p}}{2{\rho_\De}}^{2-3\chi(\mS)}\!\!\left(1-\frac{t}{{\rho_\De}^p}\right)^{3\chi(\mS)/2-1} \!\! \left(1+O\left(\sqrt{1-\frac{t}{{\rho_\De}^p}}\right)\right)\nonumber\\ \label{eq:singularB}
&&\!\!\!\!\!=_{t\to {\rho_\De}^p} \frac{p\, a(\mS)}{4} \parfrac{8\ga_\De}{p^{3/2}\rho_\De}^{\chi(\mS)}\left(1-\frac{t}{{\rho_\De}^p}\right)^{3\chi(\mS)/2-1}  \left(1+O\left(\sqrt{1-\frac{t}{{\rho_\De}^p}}\right)\right),
\end{eqnarray}
the expansion being valid in a domain dented at $t= {\rho_\De}^p$. \\

Applying standard techniques (see
\cite{FlajoletSedgewig:analytic-combinatorics}) one can obtain the
asymptotic behaviour of the coefficient
$[z^{np+2\chi(\mS)}]\gA^\De_\mS(z)= [t^n]\gB^\De_\mS(t)$ from the
singular behaviour of the series $\gB^\De_\mS(t)$:
\begin{eqnarray}
\displaystyle [t^n]\gB^\De_\mS(t)&\!\!\! =_{n\to \infty}\!\!&\displaystyle \frac{p\, a(\mS)}{4\Gamma(1\!-\!3\chi(\mS)/2)}\parfrac{8\ga_\De}{p^{3/2}\rho_\De}^{\chi(\mS)}~n^{-3\chi(\mS)/2}~{\rho_\De}^{-np}\left(1+O\left(n^{-1/2}\right)\right) \nonumber\\
&\!\!=_{n\to \infty}\!\!\!&\displaystyle \frac{p\, a(\mS)\left(8\ga_\De\rho_\De\right)^{\chi(\mS)}}{4\Gamma(1\!-\!3\chi(\mS)/2)}~(np\!+\!2\chi(\mS))^{-3\chi(\mS)/2}~{\rho_\De}^{-(np+2\chi(\mS))}\left(1+O\left(n^{-1/2}\right) \right)
\label{eq:asymptotic-B}
\end{eqnarray}

In order to conclude the proof of Theorem~\ref{thm:nb-Dangular}, it suffices to compare the number of maps in $\mM^\De_\mS(n)$ with the number of maps in $\mA^\De_\mS(n)$ and prove that 
\begin{equation}\label{eq:comparison-AandM-general}
M^\De_\mS(n)=_{n\to \infty} A^\De_\mS(n)\left(1+O\left(n^{-1/2}\right)\right).
\end{equation}
One can write a proof of Equation~\Ref{eq:comparison-AandM-general} along the line of the  proof of Lemma~\ref{lem:comparison-AandM}. We omit such a proof since a stronger statement (Theorem~\ref{thm:maps-to-dissections}) will be proved in the next section. \findem


\section{From maps to dissections}\label{section:maps-to-dissections}

In this section we prove that the number of  maps and the number of
dissections are asymptotically equivalent. More precisely, we prove
the following theorem.

\begin{thm}\label{thm:maps-to-dissections}
Let $\mS$ be a surface with boundary and let $\De\subseteq \NN^{\geq 3}$.  The number of maps in $D_\mS^\De(n)$, in $\mM_\mS^\De(n)$ and in $\mA_\mS^\De(n)$ satisfy 
$$D_\mS^\De(n)~=~ M_\mS^\De(n)\left(1+O\left(n^{-1/2}\right)\right)~=~ A_\mS^\De(n)\left(1+O\left(n^{-1/2}\right)\right).$$
\end{thm}

By Theorem~\ref{thm:maps-to-dissections}, the asymptotic enumeration of $\De$-angular maps given by Theorem~\ref{thm:nb-Dangular} also applies to $\De$-angular dissections. The rest of this section is devoted to the proof of Theorem~\ref{thm:maps-to-dissections}. \\

The inequalities $D_\mS^\De(n)\leq M_\mS^\De(n)\leq A_\mS^\De(n)$ are obvious, hence it suffices to prove 
\begin{equation}\label{eq:A-to-dissections}
A_\mS^\De(n)-D_\mS^\De(n)= O\left(\frac{A_\mS^\De(n)}{\sqrt{n}}\right).
\end{equation}
For this purpose, we will give a sufficient condition for a map $A$ in $\mA^\De_\mS$ to be the dual of a dissection in $D_\mS^\De(n)$.\\

Let $\tau^\bu$ be a doubly-rooted tree and let $e_1,\ldots,e_k$ be some edges appearing in this order on the path from the root-leaf to the marked leaf. One obtains $k+1$ doubly-rooted trees  $\tau_0^\bu,\ldots, \tau_k^\bu$ by \emph{cutting} the edges $e_1,\ldots,e_k$ in their middle (the middle point of $e_i,i=1\ldots k$ corresponds to the marked vertex of $\tau_{i-1}^\bu$ and to the root-vertex of $\tau_i^\bu$); see Figure~\ref{fig:tree-chain}. A doubly-rooted tree $\tau^\bu$ is said \emph{balanced} if there exist three edges $e_1,e_2,e_3$  on the path from the root-leaf to the marked leaf such that cutting at these edges gives 4 doubly-rooted trees which are all two-sided. For instance, the tree at the left of Figure~\ref{fig:tree-chain} is balanced.

\begin{figure}[ht!]\begin{center} \input{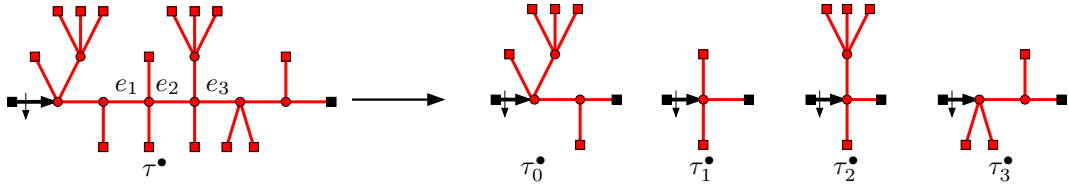}\caption{Cutting the tree $\tau^\bu$ at the edges $e_1,e_2,e_3$.}\label{fig:tree-chain} \end{center}
\end{figure}

\begin{lemma}\label{lem:balanced}
Let $\mS$ be a surface with boundary, let $A$ be a map in $\mA_\mS^\De$ and let $(S, (\tau_1^\bu,\ldots,\tau_e^\bu),$ $(\tau_1,\ldots,\tau_v))$ be its image by the bijection $\Phi$.  If all the doubly-rooted trees $\tau_1^\bu,\ldots,\tau_e^\bu$ are balanced, then $A$ is the dual of a dissection $M$ in $\mD_\mS^\De$.
\end{lemma}

\begin{proof}
We suppose that all the trees $\tau_1^\bu,\ldots,\tau_e^\bu$ are balanced. Since the trees $\tau_1^\bu,\ldots,\tau_e^\bu$ are two-sided, the map $A$ is the dual of a map  $M$ in $\mM_\mS^\De$. We want to show that the map $M$ is a dissection and for this purpose we will examine the \emph{runs} of $A$.\\ 
A \emph{run} for the map $A$ is a sequence of vertices and edges $R=v_0,e_1,v_1,\ldots,e_k,v_k$ (where $e_i$ is an edge with endpoints $v_{i-1}$ and $v_i$ for $i=1\ldots k$)  encountered when turning around a face of $A$ from one leaf to the next (that is, $v_0$ and $v_k$ are leaves while $v_1,\ldots,v_{k-1}$ are not). Since the surface $\mS$ needs not be orientable, we consider both directions for turning around faces, so that $\bR=v_k,e_k,\ldots,v_1,e_1,v_0$ is also a run called the \emph{reverse run}. 
It is clear from the definition of duality that the vertices of $M$ are in bijection with the runs of $A$ considered up to reversing, while the faces of $M$ are in bijection with the non-leaf vertices of $A$. From these bijections, it is easy to see that the map $M$ is a dissection if and only if 
\begin{enumerate}
\item[(i)] no vertex appears twice in a run,
\item[(ii)] and for any pair of distinct runs $R,R'$ such that $R'$ is not the reverse of $R$, the \emph{intersection} $R\cap R'$ (that is, the set of vertices and edges which appear in both runs) is either empty, made of one vertex, or made of one edge and its two endpoints.
\end{enumerate}
Indeed, Condition (i) ensures that no vertex of $M$ is incident twice to the same face (equivalently, no face of $M$ is incident twice to the same vertex) and Conditions (ii) ensures that any pair of vertices of $M$ which are both incident to two faces $f,f'$ are the endpoints of an edge incident to both $f$ and $f'$ (equivalently, the intersection of two faces of $M$ is either empty, a vertex, or an edge).

Observe that the runs of any tree having no vertex of degree 2 satisfy the Conditions (i) and (ii). We now compare the runs of $A$ to the runs of some trees. For $i=1\ldots e$, we choose some edges $f_{i,1}$, $f_{i,2}$ ,$f_{i,3}$ of the doubly-rooted tree $\tau_i^\bu$ such that cutting at this edges gives 4 doubly-rooted trees  $\tau^\bu_{i,j}$, $j=1\ldots 4$ which are all two-sided.  Observe that by cutting the map $A$ \emph{at all but two} of the edges $f_{i,j}$, for $i=1\ldots e$ and $j=1,2,3$, one obtains a disjoint union of plane trees (none of which has a vertex of degree 2). Moreover, since all the trees  $\tau^\bu_{i,j}$ are two-sided, no run of the map $A$ contains more than one of the edges $f_{i,j}$. Therefore, any run of $A$ is also the run of a tree (having no vertex of degree 2). Hence, no vertex appears twice in a run of $A$ and Condition (i) holds. Similarly, any pair of intersecting runs of $A$ is a pair of runs of a tree. Thus, Condition (ii) holds for $A$.
\end{proof}

We now consider the generating function $\gTh_\De(z)$ of doubly-rooted $\De$-valent trees which are not balanced (counted by number of non-root, non-marked leaves). Since the number of non-root non-marked leaves of $\De$-valent trees is congruent to $p=\gcd(\de-2,~\de\in\De)$, there exists a generating function $\gYh_\De(t)$ such that $\gTh_\De(z)=\gYh_\De(z^p)$. 
\begin{lemma}\label{lem:expansion-unbalanced}
The generating function $\gYh_\De(t)$ of non-balanced trees is analytic in a domain dented at $t= {\rho_\De}^p$ and there is a constant $\widehat{\kappa}$ such that the expansion
$$\gYh_\De(t)=_{t\to {\rho_\De}^p} \widehat{\kappa}+O\left(\sqrt{1-\frac{t}{{\rho_\De}^p}} \right)$$
is valid in this domain.
\end{lemma}

\begin{proof}
Let $\tau^\bu$ be a non-balanced doubly-rooted tree. If the tree $\tau^\bu$ is two-sided, then there exists an edge $e_1$ on the path from the root-vertex to the marked vertex such that cutting the tree $\tau^\bu$ at $e_1$ and at the edge  $e_1'$ following $e_1$ gives three doubly-rooted trees $\tau_1^{\bu}$,  $\tau_2^\bu$, ${\tau^\bu}'$ such that $\tau_1^{\bu}$ is one-sided, $\tau_2^{\bu}$ is a tree with one leg, and ${\tau^{\bu}}'$ has no two edges $e_2'$, $e_3'$ such that cutting at these edges gives three two-sided trees.
Continuing this decomposition and translating it into generating functions gives:
$$\gTh_\De(z)\leq \gTo_\De(z)+\gTo_\De(z)\gT_{\De,1}(z)\left[\gTo_\De(z)+\gTo_\De(z)\gT_{\De,1}(z)\left[ \gTo_\De(z)+\gTo_\De(z)^2\gT_{\De,1}(z)\right] \right],$$
where $\gTo_\De(z)$ is the generating function of one-sided trees and $\gT_{\De,1}(z)$ is the generating function of trees with one leg. By performing the change of variable $t=z^p$ one gets 
\begin{equation}\label{eq:bound-Yh}
\gYh_\De(t)\leq \gYo_\De(t)+\gYo_\De(t)\gY_{\De,1}(t)\left[\gYo_\De(t)+\gYo_\De(t)\gY_{\De,1}(t)\left[ \gYo_\De(t)+\gYo_\De(t)^2\gY_{\De,1}(t)\right] \right].
\end{equation}

The number of one-sided trees having $n$ non-marked non-root leaves is $2~T(n+1)$ for $n>0$ and 1 for $n=0$. Hence, $\gTo_\De(z)=2\gT_\De(z)/z-1$ and $\gYo_\De(t)=2\gY_\De(t)-1$.  Thus, Lemma~\ref{lem:singularities-trees-general} implies that the series $\gYo_\De(t)$ is analytic and has an expansion of the form 
$$\gYo_\De(t)=_{t\to {\rho_\De}^p} \frac{2\tau_\De}{\rho_\De}+O\left(\sqrt{1-\frac{t}{{\rho_\De}^p}} \right)$$
valid in a domain dented at $t={\rho_\De}^p$. Similarly, Lemma~\ref{lem:singularity} implies that the series  $\gY_{\De,1}(t)$ 
is analytic and has an expansion of the form 
$$\gY_{\De,1}(t)=_{t\to {\rho_\De}^p} \kappa_1+O\left(\sqrt{1-\frac{t}{{\rho_\De}^p}} \right)$$
valid in a domain dented at $t={\rho_\De}^p$. 
The lemma follows from these expansions and Equation~\Ref{eq:bound-Yh}.
\end{proof}

We are now ready to bound the number of maps in $\mA^\De_\mS(n)$ which are not the dual of a dissection and prove Equation~\Ref{eq:A-to-dissections}.\\

Let $\mAh^\De_\mS$ be the class of maps in $\mA^\De_\mS$ which are not the dual of a dissection in $\mD^\De_\mS$ and let $\gAh^\De_\mS(z)=\gA^\De_\mS(z)-\gD^\De_\mS(z)$ be the corresponding generating function. For any scheme $S$ of $\mS$, we also define $\gFh_S^\De(z)$ as the generating function of maps in $\mAh^\De_\mS$ which have scheme $S$. 
By Lemma~\ref{lem:balanced}, for any map $A$ in $\mAh^\De_\mS$, the image $\Phi(A)=(S,(\tau_1^\bu,\ldots,\tau_e^\bu),(\tau_1,\ldots,\tau_v))$ is such that  one of the $e$ doubly-rooted trees $\tau_1^\bu,\ldots,\tau_e^\bu$ is not balanced. 
Hence,
\begin{equation}\label{eq:comparison-F-Fh}
\gFh_S^\De(z)~\leq~ e\, \frac{\gTh_\De(z)}{\gT_\De'(z)}\, \gF_S^\De(z),
\end{equation}
where $\gT_\De'(z)$ is the generating function of doubly-rooted trees and $\gTh_\De(z)$ is the generating function of non-balanced doubly-rooted  $\Dee$-valent trees.\\

By summing~\Ref{eq:comparison-F-Fh} over all schemes $S$ of $\mS$, one gets
\begin{equation}\label{eq:comparison-A-Ah}
\gAh^\De_\mS(z)=\sum_{S \textrm{ scheme}}\gFh_S^\De(z)\leq  (2-3\chi(\mS))~\frac{\gTh_\De(z)}{\gT_\De'(z)}\sum_{S \textrm{scheme}}\gF_S^\De(z)~=~ (2-3\chi(\mS))\,\frac{\gTh_\De(z)}{\gT_\De'(z)}\gA^\De_\mS(z),
\end{equation}
since $(2-3\chi(\mS))$ is the maximal number of edges of schemes of $\mS$. 
Plugging these identity into~\Ref{eq:comparison-A-Ah} and replacing $z^p$ by $t$ gives
$$\displaystyle \gBh^\De_\mS(t)\leq \frac{(2-3\chi(\mS))\gYh_\De(t)}{\gZ_\De(t)}\gB^\De_\mS(t),$$
where $\gBh^\De_\mS(t)$ is defined by $\displaystyle \gAh^\De_\mS(z)=z^{2\chi(\mS)}\gBh^\De_\mS(z^p)$ and $\gZ_\De(t)$ is defined by $\gT_\De'(z)=\gZ_\De(z^p)$.\\

Given the expansion of $\gZ_\De(t)$ given by Equation~\Ref{eq:expansionX} and the expansion of $\gYh_\De(t)$ given by Lemma~\ref{lem:expansion-unbalanced}, one obtains the expansion   
$$\gBh^\De_\mS(t)=_{t\to{\rho_\De}^p} O\left(\sqrt{1-\frac{t}{{\rho_\De}^p}} \, \gB^\De_\mS(t)\right)$$
in a domain dented at  $t= {\rho_\De}^p$. By the classical transfer theorems, between singularity types and coefficients asymptotics, one obtains 
$$[t^n]\gBh^\De_\mS(t)=_{n\to \infty}O\left(\frac{[t^n]\gB^\De_\mS(t)}{\sqrt{n}}\right).$$
This completes the proof of Equation~\Ref{eq:A-to-dissections} and Theorem~\ref{thm:nb-Dangular}.
\findem


\section{Limit laws}\label{section:limit-laws}
In this section, we study the limit law of the number of structuring edges in $\De$-angular dissections. Our method is based on generating function manipulation allied with the so-called \textit{method of moments} (see for instance \cite{Billingsley:probability}). We first explain our method in Subsection~\ref{subsection:preliminaries-limit-laws} and then apply it in Subsection~\ref{subsection:nb-structuring} in order to determine the limit law of the number of structuring edges of $\De$-angular dissections.

\subsection{A guideline for obtaining the limit laws}\label{subsection:preliminaries-limit-laws}
We start with some notations. The expectation of a random variable
$\rX$ is denoted by $\EE{\rX}$. The $r$th  \emph{moment} of $\rX$ is
$\Ee{\rX^r}$, and the $r$th  \emph{factorial moment} is
$\EE{(\rX)_r}:=\EE{\rX(\rX-1)\cdots(\rX-r+1)}$. For a sequence
$(\rX_n)_{n\in\NN}$ of real random variables, we denote by $\rX_n\td
\rX$ the fact that $\rX_n$ converges to $\rX$ in distribution. We use the following
sufficient condition for convergence in distribution.

\begin{lemma}[Method of moments]\label{lem:method-moments}
Let $\rX_n$, $n\in\NN$ and $\rX$ be real random variables satisfying:
\begin{itemize}
\item[(A)] there exists $R>0$ such that  $\displaystyle \frac{R^r\, \EE{\rX^r}}{r!}~\to_{r\to\infty}~ 0$,
\item[(B)] for all $r\in \NN$, $\displaystyle \EE{\rX_n^r} ~\to_{n\to\infty}~ \EE{\rX^r}$.
\end{itemize}
Then $\rX_n\td \rX$.
\end{lemma}


Lemma~\ref{lem:method-moments}, leads to study the moments of random
variables. These can be accessed through generating functions in the
following way.  Suppose that $\mC$ is a \emph{combinatorial class},
that is, a set supplied with a \emph{size function}
$|\cdot|:\mC\to \NN$ such that for all $n\in\NN$ the set $\mC(n)$ of
objects of size $n$ is finite. We denote by $\rU(\mC,n)$ the random
variable  corresponding to the value of a parameter $U:\mC \to \RR$
for an element $C\in\mC$ chosen uniformly at random among those of
size $n$ (we suppose here that $\mC(n)\neq\emptyset$).
If the parameter $U$ is integer valued, then the probability
$\prob{\rU(\mC,n)=k}$ is $\displaystyle \frac{[z^n u^k]\gC(u,z)}{[z^n]\gC(1,z)}$,
where 
$$\gC(u,z)\equiv\sum_{C\in\mC}z^{|C|}u^{U(C)}$$ 
is the bivariate
generating function \emph{associated to the parameter} $U$. Thus,
the factorial moments of $\rU(\mC,n)$ can be expressed in terms of
the generating function $\gC(u,z)$:
\begin{equation}\label{eq:factorial-moments}
\EE{\pare{\rU(\mC,n)}_r}~=~\frac{\sum_{k}k(k-1)\cdots(k-r+1)[z^nu^k]\gC(u,z)}{[z^n]\gC(1,z)} ~=~\frac{[z^{n}]\pdu{r}{\gC(u,z)}}{[z^{n}]\gC(1,z)}.
\end{equation}
One can then obtain the moments of $\rU(\mC,n)$ from its factorial moments, but in our case we will consider rescaled random variables for which the following lemma applies.

\begin{lemma}\label{lemma:method-moments2}
Let $\mC$ be a combinatorial class and let $\rU(\mC,n)$ and
$\gC(u,z)$ be respectively the random variable and generating
function associated to a parameter $U:\mC\to \NN$. Let also
$\theta:\NN\to\NN$ be a function converging to $+\infty$.
If a random variable $\rX$ satisfies Condition (A) 
in Lemma~\ref{lem:method-moments} and 
\begin{itemize}
\item[(B')] for all $r\in \NN$, $\displaystyle \frac{[z^{n}]\pdu{r}{\gC(u,z)}}{\theta(n)^r[z^{n}]\gC(1,z)}~ \to_{n\to\infty}~ \EE{\rX^r}$,
\end{itemize}
then the rescaled random variables $\displaystyle \rX_n\equiv\frac{\rU(\mC,n)}{\theta(n)}$ converges to $\rX$ in  distribution.
\end{lemma}

\begin{proof}
We only need to prove that (B') implies (B). Using~\Ref{eq:factorial-moments} and the fact that $\theta$ tends to infinity gives for all $r\geq 0$,
$$\frac{[z^{n}]\pdu{r}{\gC(u,z)}}{\theta(n)^r[z^{n}]\gC(1,z)}~=~\EE{\frac{\pare{\rU(\mC,n)}_r}{\theta(n)^r}}~=~\EE{\frac{\rU(\mC,n)^r}{\theta(n)^r}}+o\pare{\sum_{k<r}\EE{\frac{\rU(\mC,n)^k}{\theta(n)^k}}}.$$
Given Condition (B'), a simple induction on $r$ shows that  $\displaystyle \frac{\rU(\mC,n)^r}{\theta(n)^r}$ has a finite limit for all $r\geq 0$. Thus,
$\displaystyle
\frac{[z^{n}]\pdu{r}{\gC(u,z)}}{\theta(n)^r[z^{n}]\gC(1,z)}\to_{n\to
\infty} \EE{\rX_n^r}$. \end{proof}

We will also use the following lemma.
\begin{lemma}\label{lem:law-subclass}
Let $\mC'$ be a subclass of the combinatorial class $\mC$ such that
$|\mC(n)|\sim |\mC'(n)|$, and let $X:\mC\to \RR$ be a parameter.
Then the sequence of random variables $\pare{\rX(\mC',n)}_{n\in\NN}$
converge in distribution if and only if the sequence
$\pare{\rX(\mC,n)}_{n\in\NN}$ does. In this case they have the same
limit.
\end{lemma}

\begin{proof} We need to prove that, for all $x\in \RR$, $\big|\prob{\rX(\mC',n)\leq x}-\prob{\rX(\mC,n)\leq x}\big|~\to_{n\to\infty} 0.$
To this end, we consider a coupling of the variables
$\rX(\mC,n)$ and $\rX(\mC',n)$ obtained in the following way. Let
$\mathfrak{C}$ and $\mathfrak{C'}$ be independent random variables
whose value is an element  chosen uniformly at random in $\mC(n)$
and $\mC'(n)$, respectively. The random variable $\mathfrak{C''}$
whose value is $\mathfrak{C}$ if $\mathfrak{C}\in \mC'(n)$ and
$\mathfrak{C'}$ otherwise is uniformly random in $\mC'(n)$. Hence,
the random variable $\rX(\mC,n)$ and $\rX(\mC',n)$ have the same
distribution as $X(\mathfrak{C})$ and $X(\mathfrak{C}'')$
respectively. Thus, for all $x\in \RR$,
$$ \left|\prob{\rX(\mC',n)\leq x}-\prob{\rX(\mC,n)\leq x}\right| ~\leq~ \prob{\mathfrak{C}''\neq \mathfrak{C}} ~\leq~ \prob{\mathfrak{C}\notin \mC'(n)}~\leq~1- \frac{|\mC'(n)|}{|\mC(n)|}~\to~ 0.$$
\end{proof}

\subsection{Number of structuring edges in $\De$-angular dissections}\label{subsection:nb-structuring}
We are now ready to study the limit law of the number of structuring edges in $\De$-angular dissections. Recall that for a map $M$ in $\mM^\De_\mS$, an edge is said \emph{structuring} if either it does not separate the surface $\mS$ or if it separates $\mS$ into two parts, none of which is homeomorphic to a disc. For a map $A\in\mA^\De_\mS$, we call \emph{structuring} the edges of the submap of $A$ obtained by recursively deleting all leaves. These are the edges whose deletion either does not disconnect the map or disconnect it in two parts, none of which is reduced to a tree.  With this definition, if the maps $M\in\mM^\De_\mS$ and $A\in\mA^\De_\mS$ are dual of each other, then their structuring edges correspond by duality. We denote by $U$ the parameter corresponding to the number of structuring edges so that for $\mC\in \{\mA^\De_\mS, \mM^\De_\mS,\mD^\De_\mS\}$, the random variable $\rU(\mC,n)$ gives the number of structuring edges of a map  $C\in\mC$ chosen uniformly at random among those of size $n$. Recall that the \emph{size} of maps in $\mA^\De_\mS$ is the number of leaves, while the \emph{size} of maps in $\mM^\De_\mS$ is the number of vertices. Hence, the bivariate generating functions $\gA^\De_\mS(u,z)$, $\gM^\De_\mS(u,z)$ and $\gD^\De_\mS(u,z)$ associated to the parameter $\rU$ for the classes $\mA^\De_\mS$, $\mM^\De_\mS$ and $\mD^\De_\mS$ satisfy: $\gA^\De_\mS(z)=\gA^\De_\mS(1,z)$, $\gM^\De_\mS(z)=\gM^\De_\mS(1,z)$ and $\gD^\De_\mS(z)=\gD^\De_\mS(1,z)$. \\

Recall that the size of maps in $\mC_\mS^\De\in \{\mA^\De_\mS,
\mM^\De_\mS,\mD^\De_\mS\}$ is congruent to $2\chi(\mS)$ modulo $p=\gcd(\de-2,~ \de\in \De)$. 
Moreover, it will be shown shortly that the average number of structuring edges of maps  $\mA^\De_\mS(n)$ is $O(\sqrt{n})$. 
This leads us to consider the following rescaled random variables.
\begin{Def}
For any set $\De\subseteq \NN^{\geq 3}$ we define the rescaled
random variables $\rX(\mC_\mS^\De,n)$, for the class $\mC_\mS^\De\in
\{\mA^\De_\mS, \mM^\De_\mS,\mD^\De_\mS\}$ by
\begin{equation}\label{eq:Xn}
 \rX(\mC^\De_\mS,n)~=~\frac{\rU(\mC_\mS^\De,np+2\chi(\mS))}{\sqrt{np+2\chi(\mS)}},
\end{equation}
where $p=\gcd(\de-2,~ \de\in \De)$.
\end{Def}

We also define some continuous random variables $\rX_k$ as follows.

\begin{Def}\label{Def:Xlim}
For all non-negative integer $k$, we denote by $\rX_{k}$ the real random variable with probability density function
\begin{equation}\label{eq:Xlim}
g_k(t)=\frac{2\, t^{3k}\,
e^{-t^2}}{\Gamma\left(\frac{1+3k}{2}\right)}~ \ind,
\end{equation}
where $\ind$ is the characteristic function of the set $[0,\infty[$.
\end{Def}

\begin{thm}\label{thm:law-De-angular}
Let $\mS$ be any surface with boundary distinct from the disc, let $\De\subseteq \NN^{\geq 3}$ and let $p=\gcd(\de-2,~ \de\in \De)$. 
The sequences of random variables $\pare{\rX(\mM^\De_\mS,n)}_{n\in\NN}$ and $\pare{\rX(\mD^\De_\mS,n)}_{n\in\NN}$ corresponding respectively to
the rescaled number of structuring edges in $\De$-angular maps and
dissections both converge in distribution to the random variable $\rX_\mS^\De\equiv \parfrac{\ga_{\De}}{\rho_\De}\, \rX_{-\chi(\mS)}$ where $\ga_{\De}$ and $\rho_\De$ are the constants defined by~\Ref{eq:constants-trees}.
\end{thm}

In the case $\De=3$, one has $p=1$,  $\rho_\De=1/4$, $\ga_\De=1/2$.
Hence, by Theorem~\ref{thm:law-De-angular} the rescaled number of
structuring edges of uniformly random simplicial decompositions
$\rU(\mD_\mS,n)/\sqrt{n}$ converges to the random variable
$2\rX_{-\chi(\mS)}$ whose probability density function is
$$f(t)~=~\frac{1}{2}g_{-\chi(\mS)}\!\parfrac{t}{2}~=~\frac{1}{\Gamma\!\parfrac{1-3\chi(\mS)}{2}} \parfrac{t}{2}^{-3\chi(\mS)}\,  e^{-t^2/4}~ \ind,$$
The rest of this section is devoted to the proof of Theorem~\ref{thm:law-De-angular}.\\

By duality, the classes $\mM_\mS^\De$ and $\mD_\mS^\De$ can be considered as subclasses of $\mA_\mS^\De$. Moreover, by Lemma~\ref{thm:maps-to-dissections}, $|\mM_\mS^\De(n)|\sim |\mD_\mS^\De(n)|\sim  |\mA_\mS^\De(n)|$. Thus, by Lemma~\ref{lem:law-subclass} it is sufficient to prove that the rescaled random variable $\rX_n(\mA_\mS^\De,n)$ converge to $\rX_\mS^\De$ in distribution.\\

In order to apply Lemma~\ref{lemma:method-moments2}, we first check that the variable $\rX_\mS^\De$ satisfy condition (A).

\begin{lemma}\label{lem:X-moments}
The  random variable  $\rX_\mS^\De$ satisfies condition (A) and its
$r$th moment is
\begin{equation}\label{eq:X-moment}
\displaystyle \EE{\left(\rX_\mS^\De\right)^r}~=~\parfrac{\ga_{\De}}{\rho_\De}^r\, \frac{\Gamma\!\parfrac{r+1-3\chi(\mS)}{2}}{\Gamma\!\parfrac{1-3\chi(\mS)}{2}}.
\end{equation}
\end{lemma}

\begin{proof} By definition of $\rX_\mS^\De$, 
$\EE{\left(\rX_\mS^\De\right)^r}=\parfrac{\ga_{\De}}{\rho_\De}^r\EE{\left(\rX_{-\chi(\mS)}\right)^r}$
where $\rX_k$ is defined by (\ref{eq:Xlim}). The moments of $\rX_k$
can be calculated by making the change of variable $u=t^2$:
$$\EE{{\rX_k}^r}= \frac{2}{\Gamma\!\parfrac{1+3k}{2}}\int_{0}^{\infty}t^{r+3k}e^{-t^2} dt= \frac{1}{\Gamma\!\parfrac{1+3k}{2}}\int_{0}^{\infty}u^{(r+3k-1)/2}e^{-u} du=\frac{\Gamma\!\parfrac{r+1+3k}{2}}{\Gamma\!\parfrac{1+3k}{2}}.$$
Thus, (\ref{eq:X-moment}) holds. Moreover, since
$\Gamma\!\parfrac{r+1+3k}{2} \leq \Gamma(r+1)=r!$ for $r$ large enough,
Condition (A) holds for any $R$ less than $\parfrac{\rho_\De}{\ga_{\De}\sqrt{p}}$.
\end{proof}

We now study the moments of the random variables $\rU(\mA_\mS^\De,n)$ corresponding to the number of structuring edges. For this purpose, we shall exploit once again the decomposition $\Phi$ of the maps in $\mA_\mS^\De$ (Figure~\ref{fig:extension-Phi}).  This decomposition leads us to consider the \emph{spine edges} of doubly-rooted trees, that is, the edges on the path from the root-leaf to the marked leaf. Indeed, if the image of a map $A\in\mA_\mS^\De$ by the decomposition $\Phi$ is $(S,(\tau_1^\bu,\ldots,\tau_e^\bu),(\tau_1,\ldots,\tau_v))$, where the doubly-rooted trees $\tau_1^\bu,\ldots,\tau_{e-1}^\bu$  correspond to the $e-1$ non-root edges of the scheme $S$, then the structuring edges of the map $A$ are the spine edges of the  doubly-rooted trees $\tau_1^\bu,\ldots,\tau_{e-1}^\bu$. \\

We denote by $V$ be the parameter corresponding to the number of spine edges and by
$$\gT_\De^\bu(u,z)=\sum_{\tau^\bu}u^{V(\tau^\bu)}z^{|\tau^\bu|}$$
the associated bivariate generating function (here the sum is over
all doubly-rooted $\Dee$-valent trees and $|\tau^\bu|$ is the number
of leaves which are neither marked  nor the root-leaf).  The
decomposition of doubly-rooted trees into a sequence of trees with
one leg by the decomposition represented in Figure
\ref{fig:decomposition-trees} shows that
\begin{equation}\label{eq:T(u,z)}
\gT_\De^\bu(u,z)=\frac{u}{1-u \gT_{\De,1}(z)}=\pare{\frac{1}{u}-\gT_{\De,1}(z)}^{-1},
\end{equation}
where $\gT_{\De,1}(z)$ is the generating function of $\Dee$-valent trees with one leg. Moreover, plugging the expression of $\gT_{\De,1}(z)$ given by Equation~\Ref{eq:T1} gives 
\begin{equation}\label{eq:T(u,z)bis}
\gT_\De^\bu(u,z)=\pare{\frac{1}{u}-1+\frac{1}{\gT_\De'(z)}}^{-1}.
\end{equation}
\begin{figure}[ht!]\begin{center} \input{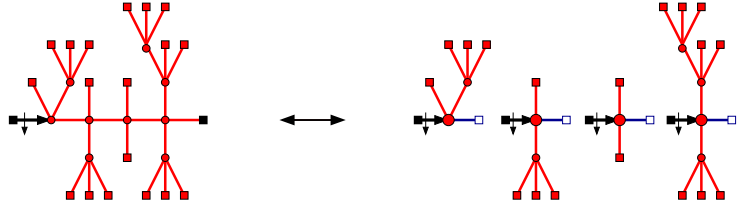}\caption{Decomposition of doubly-rooted trees as a sequence of trees with one leg.}\label{fig:decomposition-trees} \end{center}
\end{figure}

We now translate the bijection induced by $\Phi$ (Lemma
\ref{lem:decomposition-D}) in terms of generating functions. For a
scheme $S$ of  $\mS$, we denote by $\gF_S^\De(u,z)$ the generating
function of maps in $\mA_\mS^\De$ having scheme $S$ counted by
number of leaves and structuring edges. This gives
\begin{equation}\label{eq:A(u,z)}
\gA_\mS^\De(u,z)=\sum_{S \textrm{ scheme}} \gF^\De_S(u,z).
\end{equation}
The bijection  induced by $\Phi$ (Lemma~\ref{lem:decomposition-D}) and the  correspondence between spine edges and structuring edges gives
\begin{equation}\label{eq:F(u,z)}
\gF^\De_S(u,z)~=~z\, \gT'_\De(z) \left(\gT^\bu_\De(u,z)\right)^{e_S-1}\, \prod_{i=1}^{v_S} \gT_{\De,d_i(S)-1}(z),
\end{equation}
for a scheme~$S$ with $e_S$ edges and $v_S$ non-root vertices of respective degrees $d_1(S),\ldots,d_{v_S}(S)$.\\

Combining~\Ref{eq:T(u,z)bis},~\Ref{eq:A(u,z)} and~\Ref{eq:F(u,z)} gives
\begin{equation}\label{eq:A(u,z)bis}
\gA_\mS^\De(u,z)~=~z\, \gT'_\De(z) \sum_{S \textrm{ scheme}}\pare{\frac{1}{u}-1+\frac{1}{\gT_\De'(z)}}^{1-e_S}\, \prod_{i=1}^{v_S} \gT_{\De,d_i(S)-1}(z),
\end{equation}
where the sum is over the schemes $S$ of $\mS$ having  $e_S$ edges and $v_S$ non-root vertices of respective degree $d_1(S),\ldots d_{v_S}(S)$. 
Making the change of variable $t=z^p$  gives
\begin{equation}\label{eq:B(u,t)}
\gB_\mS^\De(u,t)~=~\gZ_\De(t) \sum_{S \textrm{
scheme}}\pare{\frac{1}{u}-1+\frac{1}{\gZ_\De(t)}}^{1-e_S}\,
\prod_{i=1}^{v_S} \gY_{\De,d_i(S)-1}(t),
\end{equation}
where  $\gB_\mS^\De(u,t)$ is defined by $\gA_\mS^\De(u,z)=z^{2\chi(\mS)}\gB_\mS^\De(u,z^p)$ and  $\gZ_\De(t)\equiv pt\gY'_\De(t)+\gY_\De(t)$ satisfies $\gT_\De'(z)=\gZ_\De(z^p)$.\\

By differentiating~\Ref{eq:B(u,t)} with respect to the variable $u$ and keeping only the dominant part in the asymptotic $(t,u)\to ({\rho_\De}^p,1)$ (recall that $\gZ_\De(t)\to_{t\to {\rho_\De}^p} \infty$ by~\Ref{eq:asymptW}) one gets,
$$\pduu{r}\gB_\mS^\De(u,t)=_{t\to {\rho_\De}^p}\gZ_\De(t) \sum_{S \textrm{ scheme}}\frac{(e_S+r-2)!}{u^{2r}(e_S-2)!}
\pare{\frac{1}{u}-1+\frac{1}{\gZ_\De(t)}}^{1-r-e_S}\, \prod_{i=1}^{v_S} \gY_{\De,d_i(S)-1}(t),$$
and finally
$$
\displaystyle \pdu{r}{\gB_\mS^\De(u,t)}\sim_{t\to {\rho_\De}^p}
\sum_{S \textrm{ scheme}}\frac{(e_S+r-2)!}{(e_S-2)!}\,{\gZ_\De(t)}^{e_S+r}\,
\prod_{i=1}^{v_S} \gY_{\De,d_i(S)-1}(t).
$$
The singular expansion of the series $\gY_{\De,\ell}(t)$ and $\gZ_\De(t)$  given by Lemma~\ref{lem:singularity} and Equation~\Ref{eq:asymptW} gives,
$$
\displaystyle \pdu{r}{\gB_\mS^\De(u,t)}\!\!\sim_{t\to {\rho_\De}^p}
\!\! \sum_{S \textrm{ scheme}}\!\!
\frac{(e_S+r-2)!}{(e_S-2)!}\parfrac{\ga_\De\sqrt{p}}{2\rho_\De}^{e_S+r}\!
\pare{1-\frac{t}{{\rho_\De}^p}}^{-\frac{e_S+r}{2}}\prod_{i=1}^{v_S}
\ka_{d_i(S)-1}.
$$
Since the maximum number of edges $e_S$ of a scheme
$S$ of $\mS$ is $2-3\chi(\mS)$, with equality only for the $a(\mS)$
cubic schemes, and since that cubic schemes have $v_S=1-2\chi(\mS)$ non-root vertices, one gets
\begin{equation}\label{eq:asympt-bivariate}
\displaystyle\! \pdu{r}{\gB_\mS^\De(u,t)}\!\!\!\sim_{t\to
{\rho_\De}^p}  a(\mS)\,{\ka_2}^{1-2\chi(\mS)}\, \frac{(r-3\chi(\mS))!}{(-3\chi(\mS))!}
\parfrac{\ga_\De\sqrt{p}}{2\rho_\De}^{r+2-3\chi(\mS)}\!
\pare{1-\frac{t}{{\rho_\De}^p}}^{-\frac{r+2-3\chi(\mS)}{2}}\!\!\!.\!
\end{equation}

The generating function $\pdu{r}{\gB_\mS^\De(u,t)}$ is analytic in a
domain dented at $t=\rho_\De^{p}$. Hence, the asymptotic expansion~\Ref{eq:asympt-bivariate} implies
$$
[t^n]\pdu{r}{\gB_\mS^\De(u,t)}\sim_{n\to \infty} a(\mS)\, \ka_2^{1-2\chi(\mS)}\,\frac{(r-3\chi(\mS))!}{(-3\chi(\mS))!}\parfrac{\ga_\De\sqrt{p}}{2\rho_\De}^{r+2-3\chi(\mS)}\frac{n^{\frac{r-3\chi(\mS)}{2}}{\rho_\De}^{-np}}{\Gamma(\frac{r+2-3\chi(\mS)}{2})}.
$$
Using $\ka_2=\parfrac{\rho_\De}{\ga_\De}^2$ and the asymptotic of $[t^n]\gB_\mS^\De(t)$ given by~\Ref{eq:asymptotic-B} one gets
$$\frac{[t^n]\pdu{r}{\gB_\mS^\De(u,t)}}{(np+2\chi(\mS))^{r/2}[t^n]\gB_\mS^\De(t)} \to_{n\to \infty} \parfrac{\ga_\De}{2\rho_\De}^{r}\frac{(r-3\chi(\mS))!}{(-3\chi(\mS))!}\frac{\Gamma\!\parfrac{2-3\chi(\mS)}{2}}{\Gamma\!\parfrac{r+2-3\chi(\mS)}{2}}.$$
The right-hand-side of this equation can be simplified by writing
$\displaystyle\frac{(r-3\chi(\mS))!}{(-3\chi(\mS))!}=
\frac{\Gamma(r+1-3\chi(\mS))}{\Gamma(1-3\chi(\mS))}$ and using \emph{Gauss
duplication formula}, which states that for all $x$, $\displaystyle \frac{\Gamma(x)}{\Gamma((x+1)/2)}=\frac{2^x}{2\sqrt{\pi}}\,
\Gamma(x/2)$. This gives
$$
\frac{[t^n]\pdu{r}{\gB_\mS^\De(u,t)}}{(np+2\chi(\mS))^{r/2}[t^n]\gB_\mS^\De(t)}
\to_{n\to \infty} \parfrac{\ga_\De}{\rho_\De}^{r}
\frac{\Gamma\!\parfrac{r+1-3\chi(\mS)}{2}}{\Gamma\!\parfrac{1-3\chi(\mS)}{2}}.
$$
Comparing this expression with the $r$th moment of $\rX_\mS^\De$ (Lemma~\ref{lem:X-moments}) gives
$$\frac{[z^{np+2\chi(\mS)}]\pdu{r}{\gA_\mS^\De(u,z)}}{\pare{np+2\chi(\mS)}^{r/2}[z^{np+2\chi(\mS)}]\gA_\mS^\De(1,z)}=\frac{[t^{n}]\pdu{r}{\gB_\mS^\De(u,t)}}{\pare{np+2\chi(\mS)}^{r/2}[t^{n}]\gB_\mS^\De(t)} ~\to_{n\to\infty}~ \EE{\pare{\rX_\mS^\De}^r},
$$
which is exactly Condition (B') for the convergence of $\rX(\mC^\De_\mS,n)\equiv\frac{\rU(\mC_\mS^\De,np+2\chi(\mS))}{\sqrt{np+2\chi(\mS)}}$ to $\rX^\De_\mS$. Theorem~\ref{thm:law-De-angular} then follows from Lemmas~\ref{lemma:method-moments2} and~\ref{lem:law-subclass}. 
\findem


\newpage

\appendix

\section{The smooth implicit-function schema and the counting of trees}\label{appendix:asymptotic-method}
In this section we prove Lemma~\ref{lem:singularities-trees-general} by following the methodology of  \cite[Example VII.13]{FlajoletSedgewig:analytic-combinatorics}. The main machinery is provided by a Theorem of Meir and Moon \cite{Meier:asymptotic-method} (which appears as  Theorem  VII.3 of \cite{FlajoletSedgewig:analytic-combinatorics}), on the singular behaviour of generating functions defined by a \emph{smooth implicit-function schema}.

\begin{Def}
Let $\gW(t)$ be a function analytic at 0, with $\gW(0)=0$ and
$[t^n]\gW(t)\geq 0$ for all $n\geq 0$. The function is said to
satisfy a \emph{smooth implicit-function schema} if there exists a
bivariate power series $G(t,w)=\sum_{m,n\geq 0}g_{m,n}t^m w^n$ satisfying $\gW(t) = G(t,\gW(t))$ and the following conditions:
\begin{enumerate}
\item[$(a)$] There exist positive numbers $R,S>0$ such that $G(t,w)$ is analytic in the domain $|t| < R$ and $|w| < S$.
\item[$(b)$] The coefficients of $G$ satisfy $g_{m,n} \geq  0,~ g_{0,0}=0,~ g_{0,1}\neq 1$, $g_{m,n}>0$ for some $m\geq 0$ and some $n \geq 2$.
\item[$(c)$] There exist two numbers $r,~s$, such that $0 < r < R$ and $0 < s < S$, satisfying the system of equations
$$G(r, s)=s~ \textrm{ and }~G_w(r, s)=1,$$
which is called the \emph{characteristic system} (where $G_w$ denotes the derivative of $G$ with respect to its second variable).
\end{enumerate}
\end{Def}

Recall that a series $\gW(t) = \sum_{n\geq 0}w_nt^n$ is said
\emph{aperiodic} if there exists integers $i<j<k$ such that the
coefficients of $w_i$, $w_j$,  $w_k$ are non-zero and
$\mathrm{gcd}(j-i,k-i)=1$.

\begin{lemma}[\cite{Meier:asymptotic-method}]\label{lem:asymptotic-method}
Let $\gW(t)$ be an aperiodic function satisfying the smooth implicit-function schema defined by $G(t,w)$ and let $(r,s)$ be the positive solution of the characteristic system. 
Then, the series $\gW(t)$   is analytic in a domain dented at  $t=r$. Moreover, at any order $n\geq 0$, an expansion of the form 
$$\gW(t)~=_{t\to r}~ \sum_{k=0}^n\al_k(1-t/r)^{k/2}+o\left((1-t/r)^{n/2}\right),$$
is valid in this domain, with $\al_0= s$ and $\al_1=-\sqrt{ \frac{2rG_t(r,s)}{G_{w,w}(r,s)}}$.
\end{lemma}

\begin{figure}[ht!]\begin{center} \input{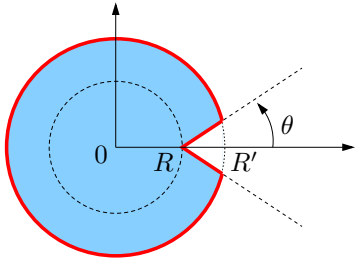}\caption{A domain dented at $R$ (dashed region).}\label{fig:dented-domain} \end{center}
\end{figure}

\begin{proof}[Proof of Lemma~\ref{lem:singularities-trees-general}] 
We first check that Lemma~\ref{lem:asymptotic-method} applies to the series $\gW(t)=\gY_\De(t)-1$.
\begin{itemize}
\item Clearly, the generating function $\gW(t)$ is analytic at $0$, has non-negative coefficients and $\gW(0)=0$. Moreover,  by Lemma~\ref{lem:periodicity}, $\gW(t)=G(t,\gW(t))$ for $G(t,w)=\sum_{k\in K}t^k(w+1)^{kp+1}$.
\item We now check that $\gW(t)$ is aperiodic. Since $p=\gcd(\de-2,~\de\in\De)$, there exist $k,l\in K$ such that $\gcd(k,l)=1$ (this includes the case $k=l=1$). It is easy to see that there exists $\Dee$-valent trees such with $(\al k+\be l)p+1$ non-root leaves for all $\al,\be\leq 0$. Hence, $[t^{\al k+\be l}]\gW(t)\neq 0$ for all $\al,\be>0$. This shows that the generating function $\gW(t)$ is aperiodic.
\item The conditions $(a)$ and $(b)$ clearly holds for $R=S=1$.  The condition $(c)$ holds for $r={\rho_\De}^p$ and $s=\tau_\De/\rho_\De-1$. Indeed with these values
$$G(r,s)=\sum_{k \in K}r^k(s+1)^{kp+1}=\sum_{k \in K}{\rho_\De}^{kp}\parfrac{\tau_\De}{\rho_\De}^{kp+1}=\frac{1}{\rho_\De}\sum_{\de \in \De}{\tau_\De}^{\de-1}=\frac{\tau_\De-\rho_\De}{\rho_\De}=s$$
and
$$G_w(r,s)=\sum_{k \in K}(kp+1)r^k(s+1)^{kp}=\sum_{k \in K}(kp+1)\tau_\De^{kp}=\sum_{\de \in \De}(\de-1)\tau_\De^{\de-2}=1.$$
\end{itemize}
Thus, the series $\gW(t)=\gY_\De(t)-1$ satisfies the conditions of Lemma~\ref{lem:asymptotic-method}. It only remains to show that $\displaystyle\al_1\equiv -\sqrt{\frac{2rG_t(r,s)}{G_{w,w}(r,s)}}=-\frac{\ga_\De}{\rho_\De\sqrt{p}}$.

One has
\begin{eqnarray}
G_t(r,s)&=&\sum_{k\in K}kr^{k-1}(s+1)^{kp+1} = \frac{s+1}{rp}\left(\sum_{k\in K}(kp+1)(r(s+1)^p)^k-\sum_{k\in K}(r(s+1)^p)^k\right) \nonumber\\
&=&\frac{s+1}{rp}\left(1-\frac{s}{s+1}\right)=\frac{1}{rp},\nonumber
\end{eqnarray}
and 
$$\displaystyle G_{w,w}(r,s)= \sum_{k\in K}kp(kp+1)r^k(s+1)^{kp-1}=\rho_\De\sum_{\de\in \De}(\de-1)(\de-2){\tau_\De}^{\de-3}=\frac{2{\rho_\De}^2}{{\ga_\De}^2}.$$
Hence,  $\displaystyle \al_1=-\frac{\ga_\De}{\rho_\De\sqrt{p}}$. This completes the proof of Lemma~\ref{lem:singularities-trees-general}. 
\end{proof}

\section{Determining the constants: functional equations for cubic maps.}\label{section:appendix-constants}
In this section, we give equations determining the constant  $a(\mS)$ appearing in Theorems~\ref{thm:nb-triangular} and~\ref{thm:nb-Dangular}. Recall that for any surface $\mS$ with boundary, $a(\mS)$ denotes the number of cubic schemes of $\mS$. Our method for determining $a(\mS)$ is inspired by the works of Bender and Canfield \cite{Bender:maps-orientable-surfaces} and Gao \cite{Gao:degree-restricted-map-general-surface}. Proofs are omitted.\\

We call \emph{$k$-marked near-cubic} maps the rooted maps having $k$ marked vertices distinct from the root-vertex and such that every non-root, non-marked vertex has degree 3. For any integer $g\geq 0$ we consider the orientable surface of genus $g$ without boundary $\TT_g$ (having Euler characteristic $\chi(\mS)=2-2g$). For any integer $k\geq 0$, we denote by $\mN_{g,k}$ the set of $k$-marked near-cubic maps on $\TT_g$ and we denote by $\gN_{g,k}(x,x_1,x_2,\ldots,x_k)\equiv\gN_{g,k}(z,x,x_1,x_2,\ldots,x_k)$ the corresponding generating function. More precisely,
$$\gN_{g,k}(x,x_1,x_2,\ldots,x_k)=\sum_{M\in \mN_{g,k}}x^{d(M)}x_1^{d_1(M)}\ldots x_k^{d_k(M)}z^{e(M)},$$
where  $e(M)$ is the number of edges, $d(M)$ is the degree of the root-vertex and $d_1(M),\ldots,d_k(M)$ are the respective degrees of the marked vertices (for a natural canonical order of the marked vertices that we do not explicit here).

Similarly, we consider the non-orientable surface of  genus $g$  without boundary $\TTh_g$ (having Euler characteristic $\chi(\mS)=2-g$). We denote by $\mNh_{g,k}$ the set of $k$-marked near-cubic maps on $\TTh_g$ and we denote by $\gNh_{g,k}(x,x_1,x_2,\ldots,x_k)\equiv\gNh_{g,k}(z,x,x_1,x_2,\ldots,x_k)$ the corresponding generating function. \\

Recall that for any surface $\mS$ with boundary, the cubic schemes of $\mS$ have $e=2-3\chi(\mS)$ edges so that 
\begin{equation}\label{eq:extract-a}
a(\mS)~=~\left\{
\begin{array}{ll}
\displaystyle [x^1z^{2-3\chi(\mS)}]\gN_{1-\chi(\mSb)/2,0}(x,z) &\textrm{ if the surface } \mS \textrm{ is orientable,} \\[3pt]
\displaystyle[x^1z^{2-3\chi(\mS)}]\gNh_{2-\chi(\mSb),0}(x,z) &\textrm{ otherwise.}
\end{array}
\right.
\end{equation}

We now give a system of functional equation determining the series
$\gN_{g,k}$ and  $\gNh_{g,k}$ uniquely.
\begin{prop}\label{prop:functionnal-eq}
The series $(\gN_{g,k})_{g,k\in \NN}$ are completely determined (as
power series in $z$ with polynomial coefficient in
$x,x_1,\ldots,x_k$) by the following system of equations:
\begin{eqnarray}\label{eq:nb-schemes}
\gN_{g,k}(x,x_1,\ldots,x_k)&=&c_0~+~\frac{z}{x}\pare{\gN_{g,k}(x,x_1,\ldots,x_k)-c_0-x[x^1]\gN_{g,k}(x,x_1,\ldots,x_k)}\nonumber \\
&& +\frac{x x_k z}{x-x_k}\pare{x\gN_{g,k-1}(x,x_1,\ldots,x_{k-1})-x_k\gN_{g,k-1}(x_k,x_1,\ldots,x_{k-1})}\nonumber \\
&& +x^2z\sum_{i=0}^g\sum_{j=0}^k\gN_{i,j}(x,x_1,\ldots,x_j)\gN_{g-i,k-j}(x,x_{j+1},\ldots,x_k)\nonumber \\
&&+ x^3z\sum_{j=1}^{k+1} \pare{\frac{\partial}{\partial x_j} \gN_{g-1,k+1}(x,x_1,\ldots,x_{k+1})} \bigg|_{x_j=x}\bigg|_{x_{j+1}=x_j}\ldots\bigg|_{x_{k+1}=x_k},
\end{eqnarray}
where $c_0=[x^0]\gN_{g,k}(x,x_1,\ldots,x_k)$ is equal to 1 if $g=k=0$ and 0 otherwise.

Similarly, the series $\gNt_{g,k}=\gN_{g/2,k}+\gNh_{g,k}$  (where $\gN_{g/2,k}$ is 0 if $g$ is odd) are determined by the following system of equations:
\begin{eqnarray}\label{eq:nb-schemes2}
\gNt_{g,k}(x,x_1,\ldots,x_k)&=&c_0~+~\frac{z}{x}\pare{\gNt_{g,k}(x,x_1,\ldots,x_k)-c_0-x[x^1]\gNt_{g,k}(x,x_1,\ldots,x_k)}\nonumber \\
&& +\frac{x x_k z}{x-x_k}\pare{x\gNt_{g,k-1}(x,x_1,\ldots,x_{k-1})-x_k\gNt_{g,k-1}(x_k,x_1,\ldots,x_{k-1})}\nonumber \\
&& +x^2z\sum_{i=0}^g\sum_{j=0}^k\gNt_{i,j}(x,x_1,\ldots,x_j)\gNt_{g-i,k-j}(x,x_{j+1},\ldots,x_k)\nonumber \\
&& + 2x^3z\sum_{j=1}^{k+1} \pare{\frac{\partial}{\partial x_j} \gNt_{g-2,k+1}(x,x_1,\ldots,x_{k+1})}\bigg|_{x_j=x}\bigg|_{x_{j+1}=x_j}\ldots\bigg|_{x_{k+1}=x_k}\nonumber \\
&& +  x^3z\frac{\partial}{\partial x} \gNt_{g-1,k}(x,x_1,\ldots,x_{k}).
\end{eqnarray}
\end{prop}

The proof of Proposition~\ref{prop:functionnal-eq} is omitted. We only indicate that the second summand in the right-hand-side of Equation~\Ref{eq:nb-schemes} (resp. Equation~\Ref{eq:nb-schemes2}) corresponds to maps in $\mN_{g,k}$ (resp.  $\mNh_{g/2,k}\cup\mNh_{g,k}$) such that the root-edge joins the root-vertex to a non-root non-marked vertex; the  third summands corresponds to maps such that the root-edge joins the root-vertex to a marked vertex;  the  fourth summand corresponds to maps such that the root-edge is a loop which separates the surface $\TT_g$ (resp. $\TTh_g$) into two connected-components; the  fifth  summand (resp. fifth and sixth summands) corresponds to maps such that the root-edge is a loop which does not separate the surface.\\

Proposition~\ref{prop:functionnal-eq} together with Equation~\Ref{eq:extract-a} give a recursive way for computing the constants $a(\mS)$ for any surface $\mS$. The first values are given for orientable surfaces in Table~\ref{table:nb-schemes-orientable} and for non-orientable surfaces and in Table~\ref{table:nb-schemes-nonorientable}. The first line in Table~\ref{table:nb-schemes-orientable} corresponds to rooted planar cubic maps with $\be$ edges. These maps where enumerated in \cite{Mullin:cubic-maps} and a nice formula exists in this case:
$$a(\mS)~=~\frac{2^{\be(\mS)}(3\be(\mS)-6)!!}{8\, \be(\mS)!(\be(\mS)-2)!!}.$$
The first column in Table~\ref{table:nb-schemes-orientable} corresponds to rooted cubic maps with a single face on the $g$-torus. These maps where first enumerated by Lehman and Walsh \cite{Walsh:counting-maps-1}. Indeed, a special case of \cite[Equation (9)]{Walsh:counting-maps-1} gives the following formula
$$a(\mS)~=~\frac{2(6g-3)!}{12^g g! (3g-2)!},$$
which was also proved bijectively in \cite{Chapuy:cubic-unicellular}.

\begin{table}[h]
\begin{center}
\begin{tabular}{l|llll}
Orientable & $\be(\mS)=1$ &$\be(\mS)=2$ &$\be(\mS)=3$ &$\be(\mS)=4$ \\\hline\vspace{-.3cm} \\
Genus 0: $\chi(\mSb)=~2$& 0 & 1 & 4 & 32 \\ 
Genus 1: $\chi(\mSb)=~0$& 1 & 28 & 664 & 14912 \\
Genus 2: $\chi(\mSb)=-2$& 105 & 8112 & 396792 & 15663360 \\
Genus 3: $\chi(\mSb)=-4$& 50050 & 6718856 & 51778972 & 30074896256\\
\end{tabular}
\caption{The number $a(\mS)$ of cubic schemes of orientable surfaces.}\label{table:nb-schemes-orientable}
\end{center}
\end{table}

\begin{table}[h]
\begin{center}
\begin{tabular}{l|llll}
Non-orientable    & $\be(\mS)=1$ &$\be(\mS)=2$ &$\be(\mS)=3$ &$\be(\mS)=4$ \\\hline\vspace{-.3cm} \\
Genus 1: $\chi(\mSb)=~1$& 1 & 9 & 118 & 1773 \\
Genus 2: $\chi(\mSb)=~0$& 6 & 174 & 4236 & 97134\\
Genus 3: $\chi(\mSb)=-1$& 128 & 6786 & 249416 & 7820190\\
Genus 4: $\chi(\mSb)=-2$& 3780 & 301680 & 15139800 & 610410600\\
\end{tabular}
\caption{The number $a(\mS)$ of cubic schemes of non-orientable surfaces.}\label{table:nb-schemes-nonorientable}
\end{center}
\end{table}

\bibliography{biblio-simplicial}
\bibliographystyle{abbrv}

\end{document}